\documentclass[11pt,a4paper]{amsart}
\usepackage{hyperref}
\usepackage{color}
\usepackage{graphicx}

\usepackage[smalltableaux, centertableaux]{ytableau}
\usepackage{relsize}
\usepackage{amsfonts,epsfig,stmaryrd}
\usepackage{,amsfonts,amssymb,amsmath,amsthm,amsgen,amscd,datetime,dsfont,hyperref,tensor}

\usepackage{multicol,picinpar,enumerate,cleveref,mathabx}
\usepackage{euscript,color}

\DeclareMathOperator{\diag}{diag}

\newcommand{\p}{\mathfrak{p}}

\newcommand{\g}{\mathfrak{g}}

\renewcommand{\l}{\mathfrak{l}}
\newcommand{\h}{\mathfrak{h}}
\newcommand{\z}{\mathfrak{z}}
\newcommand{\s}{\mathfrak{s}}
\newcommand{\m}{\mathfrak{m}}
\newcommand{\n}{\mathfrak{n}}
\renewcommand{\k}{\mathfrak{k}}

\newcommand{\+}{\oplus}

\newcommand{\x}{\vect{x}}

\newcommand{\so}{\mathfrak{so}}

\newcommand{\co}{\mathfrak{co}}

\newcommand{\gl}{\mathfrak{gl}}

\newcommand{\GL}{\mathbf{GL}}

\newcommand{\SO}{\mathbf{SO}}
\renewcommand{\O}{\mathbf{O}}
\newcommand{\CO}{\mathbf{CO}}

\newcommand{\1}{\mathbf{1}}
\newcommand{\e}{\mathrm{e}}
\newcommand{\f}{\mathfrak{f}}
\renewcommand{\d}{\mathrm{d}}

\newcommand{\End}{\mathrm{End}}

\renewcommand{\c}{\bar{c}}

\usepackage{ upgreek }

\newcommand{\R}{\mathbb{R}} 
\newcommand{\C}{\mathbb{C}} 
\renewcommand{\S}{\mathbb{S}} 

\DeclareMathOperator{\Conf}{Conf}

\DeclareMathOperator{\Iso}{\mathrm{Isom}}

\newcommand{\F}{\mathcal{F}}

\newcommand{\G}{\mathcal{G}}

\renewcommand{\H}{\mathbf H}

\newcommand{\vect}[1]{\boldsymbol{#1}} 

\def\C{\mathbb{C}}
\def\R{\mathbb{R}}

\def\F{\mathcal {F}}

\def\Iso{\sf{Iso}}

\def\B{\mathcal{B}}

\def\GL{{\sf{GL}}}
\def\O{{\sf{O}}}
\def\CO{{\sf{CO}}}
\def\SO{{\sf{SO}}}

\def\End{{\sf{End}}}

\def\Iso{{\sf{Iso}} }

\def\Ad{{\sf{Ad}} }
\def\ad{{\sf{ad}} }

\def\p{{\mathfrak{p}}}
\def\k{{\mathfrak{k}}}

\def\g{{\mathfrak{g}}}

\def\f{{\mathfrak{f}}}
\def\so{{\mathfrak{so}}}

\def\n{{\mathfrak{n}}}
\def\s{{\mathfrak{s}}}
\def\h{{\mathfrak{h}}}
\def\i{{\mathfrak{i}}}
\def\z{{\mathfrak{z}}}
\def\z{{\mathfrak{z}}}
 
\def\c{{\mathfrak{c}}}

\theoremstyle{definition}
\theoremstyle{definition}
\newtheorem{definition}{Definition}[section]
\newtheorem{remark}[definition]{Remark}

\newtheorem*{example*}{Example}
\newtheorem*{remark*}{Remark}

\theoremstyle{plain}
\newtheorem{lemma}[definition]{Lemma}
\newtheorem*{lem*}{Lemma}
\newtheorem{proposition}[definition]{Proposition}
\newtheorem{corollary}[definition]{Corollary}
\newtheorem{theorem}[definition]{Theorem}
\newtheorem*{theorem*}{Theorem}

\newtheorem*{conjecture*}{Conjecture}

\begin{document}

\title{Locally conformally  homogeneous {L}orentzian spaces}

 \author[Thomas Leistner]{Thomas Leistner}\address{School of  Mathematical Sciences, Adelaide University, SA~5005, Australia}\email{thomas.leistner@adelaide.edu.au}

 \author[Lilia Mehidi]{Lilia Mehidi}
 \address{LMBA, CNRS, Universit\'e de Bretagne Occidentale, France}
 \email{lilia.mehidi@univ-brest.fr}
 
\author[Abdelghani Zeghib]{Abdelghani Zeghib}
\address{UMPA, CNRS, \'Ecole Normale Sup\'erieure de Lyon, France}
\email{abdelghani.zeghib@ens-lyon.fr}


\maketitle
\begin{abstract}
We study locally conformally homogeneous Lorentzian manifolds of dimension at least $3$,   admitting an essential pseudo-group of local conformal transformations. Generalizing a recent result of Alekseevsky and Galaev, we show that any such manifold $(M,g)$ is either conformally flat, or locally conformally equivalent to a homogeneous plane wave.   
When the manifold is non-conformally flat, we show the existence of a codimension-one lightlike foliation of Heisenberg type, which leads to the plane wave structure. Our approach relies on tools from Gromov's theory of rigid transformations. 
Finally, we observe that the plane wave metric in the conformal class coincides with the Penrose limit of $(M,g)$ along some null geodesic.
\end{abstract}

 \tableofcontents
 
 \section{Introduction}

 A fundamental problem in conformal geometry is to classify conformal manifolds with  \textit{essential} conformal transformations, that is, global conformal transformations that do not preserve a metric in the conformal class. For Riemannian conformal manifolds the situation is very rigid: up to conformal equivalence only the round sphere and Euclidean space admit essential conformal transformations. This was conjectured by Lichnerowicz in the 1960's  and   proved in a series of papers by Ferrand and Obata (\cite{Lelong-Ferrand71},  \cite{Obata71}, \cite{Ferrand96}, with contributions in \cite{Alekseevskii72}). 
 
For indefinite metrics, the situation is far more flexible: there exist pseudo-Riemannian metrics that are not conformally flat yet admit essential conformal transformations, both on non-compact manifolds and, as shown in \cite{frances12}, on compact manifolds in signatures other than the Lorentzian one. Non-compact examples in Lorentzian signature consist of plane waves. These are manifolds for which the metric in dimension $n+2$  is locally given by
\begin{equation}\label{pw}
g=2\,\d t\, \d v + \x^\top Q(t) \x\ \d t^2 +\d\x^\top\d\x,
\end{equation}
where $\x=(x^1, \ldots , x^n)^\top\in \R^n$, {$Q(t)$ is} a $t$-dependent symmetric $n\times n$-matrix. If $Q$ is not a scalar matrix, this metric is not conformally flat.
Plane waves for which this metric is defined on all of $\R^{n+2}$, for example,   admit non-isometric homotheties with fixed points. Therefore these homotheties are essential (for this implication see for example \cite[Proposition 2.5]{LeistnerTeisseire22}). 

\bigskip

Recently 
an important rigidity result for essential Lorentzian conformal structures in the  \textit{conformally homogeneous} setting has been obtained  by Alekseevsky and Galaev
 in \cite{AlekseevskyGalaev25}. They showed that if the manifold is simply connected and admits a transitive group of global conformal transformations that is essential, then the manifold is either conformally flat or the conformal class contains a plane wave metric that is isometrically homogeneous and geodesically complete. In the present article we will generalize this result in the following sense: we will drop the assumption that the manifold is simply connected and we will only assume that the manifold is  \textit{locally conformally homogeneous}, i.e.~the local conformal transformations act transitively. We will prove the following result.

 \begin{theorem}\label{maintheo}Let $(M, g)$ be a Lorentzian manifold of dimension at least $3$ and let $\bf P$ be its pseudo-group of local conformal transformations, that is the collection 
 of all conformal diffeomorphisms between open subsets of $M$. We assume $M$ locally conformally homogeneous, i.e.~$M$ is an orbit of $\bf P$. Then one of the following possibilities holds:
 
 \begin{enumerate}
 
 \item $\bf P$ preserves a Lorentzian metric in the $g$-conformal class.
 
 \item $(M, g)$ is conformally flat. 
 
 \item $(M, g)$ is locally conformally equivalent to a plane wave that is isometrically locally homogeneous.  More precisely, there exists a homogeneous plane wave $X$, such that 
 $(M, [g])$ is modeled on $(\Conf(X), X)$.
 \end{enumerate}
 \end{theorem}

When the pseudo-group $\mathbf{P}$ does not preserve any metric in the $g$-conformal class, we say that $(M,g)$ is \textit{weakly essential}, as introduced in Section \ref{Section: terminology}. This is a weaker condition than being essential.
 Our proof starts off similarly as in \cite{AlekseevskyGalaev25} by  noting that weak essentiality yields the existence of an element in the isotropy group that is not contained in $\SO(T_pM)$.
In \cite{AlekseevskyGalaev25}, simple connectedness and global conformal homogeneity ensure that the isotropy group is connected, which in turn provides a specific element in its Lie algebra (that does not lie in the isometry algebra). In our setting, however, we cannot directly conclude the existence of such an element. Therefore, we adopt a different approach. We work only with the local isotropy of local conformal transformations, but we use powerful tools from Gromov's theory of rigid transformations (see Appendix \ref{Gromov's Theory}) and from algebraic groups, which allow us to reduce the problem to the Lie algebra level (see Section \ref{Section: isotropy}), and in particular to obtain a hyperbolic element in the isotropy algebra that is not in $\so(T_pM)$.
Since we assume only local homogeneity, we can pass to the universal cover, where weak essentiality is preserved. Moreover, on a simply connected real-analytic manifold, local conformal vector fields extend to global ones. We may therefore work with the Lie algebra of conformal vector fields on the universal cover. Loosely speaking, by weakening the assumption of homogeneity to local homogeneity, we are able to overcome  the assumption of simple connectedness. 
Then, in Section \ref{Section: Heis algebra}, we use the decomposition of the Lie algebra of conformal vector fields induced by the hyperbolic element to derive the existence of a codimension-one lightlike foliation of \textit{Heisenberg type}, i.e., with a transitive action of a Heisenberg algebra.  
Our approach, together with the new characterization of plane waves obtained in Section \ref{Section: Determining a plane wave by Killing fields}, not only leads to the generalization in Theorem \ref{maintheo}, but also provides a more conceptual approach and simplifies the proof in \cite{AlekseevskyGalaev25} for the simply connected globally homogeneous case.

 \bigskip

In the last section, we consider the consequence of our result for {\bf Penrose limits} of locally conformally homogeneous Lorentzian manifolds. 
The Penrose limit is a famous construction in gravitational physics, which shows that any Lorentzian manifold admits a  plane wave spacetime as a limit, this limit being taken along any lightlike geodesic (see \cite{Blau-PL} for instance). 
As shown in  Theorem \ref{maintheo}, a weakly essential locally conformally homogeneous Lorentzian manifold of dimension $\geq 3$ is either conformally flat or locally conformal to a homogeneous plane wave.
A natural question is then whether this plane wave coincides with the Penrose limit of the Lorentzian manifold along some lightlike geodesic. A priori, the Penrose limit is a plane wave associated with a metric, rather than with a conformal structure.
The underlying question is then whether two conformally related metrics have conformally equivalent Penrose limits along some lightlike geodesic (note that conformally related metrics have the same unparameterized lightlike geodesics). It turns out that the answer is affirmative, so the Penrose limit is a conformal invariant, as shown  in Proposition~\ref{Penrose-prop}.
As a consequence, we obtain that the plane wave metric in the conformal class of a weakly essential locally conformally homogeneous Lorentzian manifold $(M,g)$ (as in item (3) of  Theorem~\ref{maintheo}) coincides with the Penrose limit of $(M,g)$ along some null geodesic.
\begin{corollary} 
Let $(M, g)$ be a non-conformally flat Lorentzian manifold of dimension at least $3$ which is locally conformally homogeneous. Assume that $(M,g)$ is weakly essential. Then, $(M,g)$  is locally conformally equivalent to its Penrose limit along some null geodesic.  
\end{corollary}
This raises the question of whether a direct proof of this fact can be found, which we plan to investigate in future work.

\section{Terminology}\label{Section: terminology}

In this brief section we will fix our terminology for actions on manifolds. We always assume that the manifold $M$ is connected.   A   semi-Riemannian manifold $(M,g)$ is  \textit{conformally homogeneous} if the group of conformal transformations acts transitively, and  \textit{locally conformally homogeneous} if the pseudo-group of local conformal transformations acts transitively. The latter 
 means that for any $p,q \in M$, there exist neighborhoods $V_p$ of $p$ and $V_q$ of $q$, and a 
 conformal diffeomorphism  $\phi: V_p \to V_q$ sending $p$ to $q$. 
 
If $\g$ is a Lie algebra of conformal vector fields on $(M,g)$ and $p \in M$, the  \textit{$\g$-orbit of $p$} is the set of points in $M$ that can be reached from $p$ via the iterated flows of vector fields from $\g$. Then the Lie algebra $\g$ acts  \textit{transitively} if $M$ is the only $\g$-orbit. This is equivalent to the property that for any $p \in M$, the evaluation map
$X \in \g \mapsto X(p) \in T_p M$
is surjective. Indeed, if the evaluation map is surjective at each point, each $\g$-orbit is open, and since we assume that $M$ is connected, it must be equal to $M$. 
It turns out that if $(M,g)$ is locally conformally homogeneous and simply connected, any local conformal vector field extends globally on $M$, giving rise to the Lie algebra of global conformal vector fields on $M$. In this case, local conformal homogeneity is equivalent to $\g$ acting transitively (this is recalled in Appendix \ref{Gromov's Theory} and Proposition \ref{Prop. M 1-connected}).

Finally, we say that $\g$ acts  \textit{locally transitively (at $p$)} if the $\g$-orbit of $p$ is open, i.e., if the evaluation map at $p$ is surjective.

Following the usual terminology, e.g. in \cite{Frances08}, $(M,g)$ is  \textit{inessential} if there is a metric $g^\prime$ in the conformal class of $g$ 
 which is preserved by all   conformal transformations, and  \textit{essential} otherwise. For our purposes we need a local version of this definition. We say that $(M,g)$ is  \textit{strongly inessential} if there is a metric $g^\prime$ in the conformal class of $g$ 
 which is preserved by all  \textit{local}  conformal transformations, and  \textit{weakly essential} otherwise.
 A Lie algebra $\g$ of conformal vector fields is  \textit{essential} if there is no metric in the conformal class for which the vector fields from $\g$ are Killing. As will be recalled in the text, in our context of locally conformally homogeneous structures, weakly essential is equivalent to the Lie algebra $\g$ of global conformal vector fields on the universal cover being essential.

\section{Plane waves}
\subsection{Basic properties of plane waves}
\label{pwbackground-sec}
A  \textit{plane wave} is a non-flat Lorentzian manifold $(M,g)$ that admits a  vector field $\xi$ such that  \begin{equation}\label{pwdef}
\nabla \xi=0,\quad g(\xi,\xi)=0,\quad
R(X,Y)=0,\quad \nabla_XR=0\quad\text{ for all $X,Y\in \xi^\perp$,}\end{equation} 
where $\nabla$ is the Levi-Civita connection and $R$  the curvature tensor of $(M,g)$.
Since a plane wave admits a parallel null vector field $\xi$, its tangent bundle is filtered into a null line bundle  and 
the bundle $\xi^\perp$ of degenerate (null) hyperplanes,
$\langle \xi\rangle \subset \xi^\perp \subset TM$,
which are both parallel and hence provide totally geodesic  foliations of $M$. 

Locally, a plane wave metric in dimension $n+2$ is given by
\begin{equation}\label{pw}
g=2\d t \d v + \x^\top Q(t) \x\ \d t^2 +\d\x^\top\d\x,
\end{equation}
where $\x=(x^1, \ldots , x^n)^\top\in \R^n$, {$Q(t)$ is} a $t$-dependent symmetric $n\times n$-matrix {and $\xi=\partial_v$.}  

It is well-known (see for instance \cite{blau-oloughlin03}) that  a plane wave in the form~(\ref{pw}) admits a $(2n+1)$-dimensional Heisenberg algebra $\mathfrak{heis}$ of Killing vector fields that span the hyperplane distribution $\xi^\perp$. These are given by the central $\xi=\partial_v $ and
\[ \vect{u}^\top \partial_{\x}- \dot{\vect{u}}^\top\x \,\partial_v=\sum_{i=1}^n u^i \partial_{x^i} - (\sum_{i=1}^n \dot{u}^i x^i) \xi, \]
where $\vect{u}=(u^1,\ldots, u^n)$ is a solution to the ODE $\ddot{\vect{u}}=Q\vect{u}$,  see for example~\cite[Section~4.3]{GlobkeLeistner16}. 
Moreover (see for instance \cite{blau-oloughlin03}), the full algebra of Killing fields tangent to $\xi^\perp$ is an extension $\k \ltimes \mathfrak{heis}$, where $\k$ is  the Lie algebra of a compact Lie group, and consists  of Killing fields contained in the isotropy algebra of some point.

It has been shown  in \cite{blau-oloughlin03} that the metric~(\ref{pw}) admits a Killing vector field that is transversal to $\xi^\perp$, i.e.~that it has a  transitive algebra of Killing vector fields, if and only if 
\[Q(t)=\e^{tF} S \e^{-tF}\quad \text{ or }\quad Q(t)=\tfrac{1}{t^2}\e^{\log (t) F}S\e^{-\log ( t)F},\]
where $S$ is a symmetric matrix and $F$ a skew  matrix, both constant and of size $n\times n$. The first case is referred to as  \textit{regular} and the second as  \textit{singular}. In the singular case the metric is only defined  on $t>0$. 
The additional Killing vector field in these cases is 
\[\partial_t+\x^\top F\partial_{\x}\quad\text{ or }\quad t\partial_t+\x^\top F \partial_{\x}-v\partial_v,\]
respectively. Hence, if the Lie algebra of Killing vector fields acts  transitively, then it contains $\R \ltimes (\k \ltimes \mathfrak{heis})$. If the plane wave is non-flat, this is the full  algebra of Killing fields.  
Now, by \cite[Proposition 5.1]{HPW}, this algebra is isomorphic to $$\g_0 \cong (\R \oplus \k) \ltimes \mathfrak{heis}.$$ 
 Let us remark that the metrics arising in the regular and singular cases are not locally isometric; however, they are conformally equivalent, even globally if  the regular metric~(\ref{pw}) is defined on $\R^{n+2}$ and  the singular one on $\{t>0\}\times \R^{n+1}$, see for example \cite{Holland-Sparling} or \cite{AlekseevskyGalaev25}. 

The metric~(\ref{pw}) also admits a homothetic vector field 
\[2v\partial_v + \x^\top\partial_{\x},\] that commutes with $\k$, and with the additional Killing vector field.
For a non-conformally flat plane wave metric on $I \times \R^{n+1}$, where $I$ is an interval, the Lie algebra of conformal vector fields was determined in \cite[Theorem~2, Corollary~2]{Holland-Sparling} (without assuming local conformal homogeneity). If the metric is locally conformally homogeneous, it follows from \cite[Theorem~2]{Holland-Sparling} that, if the Lie algebra of conformal vector fields acts transitively, then there exists a plane wave metric in the conformal class for which the Lie algebra of Killing fields also acts transitively.  
In this case, the Lie algebra of conformal vector fields contains
$$
\g = (\R \oplus \R \oplus \k) \ltimes \mathfrak{heis} 
= \k \ltimes \mathfrak{r},
$$
where $\mathfrak{r} := (\R \oplus \R) \ltimes \mathfrak{heis}$ is contained in the solvable radical of $\g$.
Moreover, they show that, for this plane wave metric in the conformal class, the full Lie algebra of conformal vector fields is homothetic and coincides with $\g$.

\subsection{Determining a plane wave by its Killing fields}\label{Section: Determining a plane wave by Killing fields}

The fact that a plane wave is actually characterised by the  particular arrangement of Killing vector  fields described above is shown by the following result. This result is a  generalization of \cite[Theorem 3]{GlobkeLeistner16}, where it was assumed that the Killing vector fields commute.
We formulate this generalization for indefinite metrics of arbitrary signature. The Lorentzian version of it will be used in the proof of Theorem~\ref{maintheo}.
\begin{proposition}\label{pwkilling-prop}
Let $(M,g)$ be a  semi-Riemannian manifold of dimension $m=n+2$ and with Riemannian curvature tensor $R$. Assume that there are Killing vector fields $\xi_0,\xi_1,\ldots \xi_n$   that span $\xi_0^\perp$ and satisfy
\[[\xi_0,\xi_i]=0,\qquad [\xi_i ,\xi_j]\in \langle \xi_0\rangle,\qquad \text{ for all }i,j =1,\ldots n.\]
Then $\xi_0$ is parallel and it holds that $R(X,Y)=0$ and $\nabla_XR=0$ for all $X,Y\in \xi_0^\perp$. In particular, if  $(M,g)$ is Lorentzian, it is a plane wave.
\end{proposition}
\begin{proof} We set $g_{ab}:=g(\xi_a,\xi_b)$ for $a,b=0\,\ldots, n$. 
Since   $[\xi_a,\xi_b]\in \Gamma(\langle  \xi_0\rangle)$ and $\xi_a\in \xi_0^\perp$ we  have that
\[ g([\xi_a,\xi_b],\xi_c)=0\]
so that 
\[0=\mathcal L_{\xi_a} g(\xi_b,\xi_c)=\xi_a(g_{bc}),\] 
since the $\xi_a$ are Killing.
Hence the Koszul formula yields
\[
2g(\nabla_{\xi_a} \xi_b,\xi_c)=0,
\]
for all $a,b,c=0,\ldots n$, so that $\nabla_{\xi_a} \xi_b\in \Gamma(\langle \xi_0\rangle)$. Note also that $[[\xi_a,\xi_b],\xi_c]=0$.

First we show that $\xi_0$ is parallel. 

Fix an arbitrary $p\in M$ and we will show that $\nabla \xi_0|_p=0$. 
Let $Z$ be local null vector field near $p$ such that $g(Z,\xi_0)=1$,  $g(Z,\xi_i)=0$, for $i=1,\ldots, n$.  On the domain of $Z$, the 
assumptions imply that
\begin{equation}
\label{LieZ}
0=\mathcal L_{\xi_a} g(Z,\xi_b)=-g([\xi_a,Z], \xi_b)- g([\xi_a,\xi_b],Z).\end{equation}
For $a=0$ or $b=0$ this implies 
\[0=\mathcal L_{\xi_a} g(Z,\xi_b)=g([\xi_a,Z], \xi_0),\qquad  0=\mathcal L_{\xi_0} g(Z,\xi_a)= g([\xi_0,Z], \xi_a),\]
so that the Koszul formula gives
\[g(\nabla_{\xi_a}\xi_0,Z)=0,\qquad g(\nabla_{Z}\xi_0,\xi_a)=0,\]
and hence $\nabla_{\xi_a}\xi_0=0$ and $\nabla_{Z}\xi_0\in \rangle \xi_0\langle $. Finally, since $\xi_0$ is Killing $g(\nabla_{Z}\xi_0,Z)=0$, so that $\xi_0$ is parallel on the domain of $Z$ and hence everywhere. Note that, since $[\xi_a,\xi_b]=c_{ab} \xi_0$ with constants $c_{ab}$,  this also implies 
\[\nabla [\xi_a,\xi_b]=0,\]
and we also have that $[Z,\xi_0]=0$.

Next, we need to show that $R(\xi_a,\xi_b)=0$, where $R\in \Lambda^2\otimes \so(TM)$ is the Riemannian curvature tensor. 
This is equivalent to $R(\xi_a,\xi_b)\xi_c=0$, so we need to determine $\nabla_{\xi_a}\xi_b$. By~(\ref{LieZ})
and the Koszul formula we get
\[2g(\nabla_{\xi_a} \xi_b ,Z)=-Z(g_{ab}) +g([\xi_a,\xi_b],Z),\]
and consequently
\[ \nabla_{\xi_a} \xi_b =\tfrac{1}{2}\left( g([\xi_a,\xi_b],Z)-Z(g_{ab}) \right)\xi_0.\]
By $g([Z,\xi_a],\xi_0)=0$ it is $ [Z,\xi_a]\in \xi_0^\perp$, so that
\[[Z,\xi_a](g_{bc}) =0\] 
and hence $\xi_a(Z(g_{bc})=0$.
With $\xi_0$ being parallel, this implies that
\[\nabla_{\xi_a}\nabla_{\xi_b} \xi_c=\tfrac12 \xi_a \left(g([\xi_b,\xi_c],Z)\right) \xi_0
=\tfrac12 \left(g([\xi_a,[\xi_b,\xi_c]],Z)\right) \xi_0=0,
\]
because  $g([\xi_b,\xi_c],[\xi_a,Z]) =0$. In addition, with $[\xi_a,\xi_b]$ parallel, we have 
\[g(\nabla_{[\xi_a,\xi_b]}\xi_c,Z)=g([[\xi_a,\xi_b],\xi_c],Z)=0,\] so that
$
R(\xi_a,\xi_b)\xi_c = 0$ as required.

Finally, we need to show that $\nabla_{\xi_a}R=0$. This can be achieved by the well-known integrability condition for  a Killing vector field $\xi$,
\[\nabla_{\xi} R= (\nabla \xi)\cdot R ,\]
where the dot denotes the action of $\so(TM)$ on the curvature tensors. However, the above calculations have shown 
the range of $\nabla \xi |_{\xi_0^\perp}$ is a section of   $\xi_0^\perp \otimes \langle\xi_0\rangle$. Moreover, since  
\[g(\nabla_{Z}\xi_a,\xi_0) = - g(\nabla_{\xi_0}\xi_a,Z)=0,\]
we have that $\nabla_Z \xi_a$ is a section of $ \xi_0^\perp$.
Then $R(\xi_a,\xi_b)=0$ implies that $\nabla_{\xi_a}R=0$ and hence that $\nabla_XR=0$ for all $X\in \xi_0^\perp$. In particular, if $g$ is Lorentzian, then $(M,g)$ is a plane wave.
\end{proof}

In order to prove Theorem~\ref{maintheo}, in the remainder of the article we will deduce the existence of the Killing vector fields as in Proposition~\ref{pwkilling-prop} under the appropriate assumptions.

\section{Isotropy group and reductions}\label{Section: isotropy}

 In this section we will study the isotropy of a locally conformally homogeneous manifold and its relation to weak essentiality.
 Let $(M, g)$ be a   locally conformally homogeneous Lorentzian manifold of dimension 
  $n+2\ge 3$. 
 
 \subsection{The isotropy} Fix $p$, and let $H_p$ be the group of (germs) of local conformal transformations fixing $p$. Here ``germ'' means that one identifies local conformal diffeomorphisms  $f_1 $ and $f_2$,  
 $f_i: U^i_p \to V^i_p$, if they  coincide on some  $W_p \subset U^1_p \cap U^2_p$,  where $U^i_p, V^i_p, W_p$ are neighborhoods of $p$. For $H_p$ we have the    \textit{derivative  representation} or  \textit{isotropy representation} $\rho_p: f \in H_p \to \d f_p \in  \GL(T_pM)$. Clearly the  image of $\rho_p$ is contained in 
$\CO(T_pM)\simeq \CO(1,n+1)= \R^* \times \O(1,n+1)$.  The following observations will be fundamental for our approach.

\begin{proposition}\label{Prop. isotropy} If $(M,g)$ is locally conformally homogeneous, then the following holds:

\begin{enumerate}

\item\label{iness-O} $(M, g)$ is weakly essential if and only if the image of $\rho_p$ is not contained in $\O(T_pM, g_p)$.

\item If $\rho_p(H_p)$ contains $e^\lambda \1$ with $\lambda\not= 0$, then $(M,[g]))$ is conformally flat.

\item If $\rho_p$ is not injective, then 
 $(M, g)$ is conformally flat.

\item  $\rho_p(H_p)$  contains the identity component of its Zariski closure in $  \GL(T_pM) \simeq \GL(n+2,\R)$, contained in $\CO(1,n+1)$.
\end{enumerate}
\end{proposition}

\begin{proof}
\begin{enumerate}
\item If $\rho_p(H_p) \subset \O(T_pM)$, one can define a metric as follows. By local homogeneity, for any $q \in M$ there exists a local conformal transformation  $\phi: V_p\to V_q $ between open neighborhoods of $p$ and $q$, such that $\phi(p)=q$. We then  define $\hat g_q(\d \phi_p(X),\d \phi_p Y):=g_p(X,Y)$ for $X$ and $Y$ in $T_pM$. Since $\rho_p(H_p) \subset \O(T_pM)$, this definition is independent of the choice of $\phi$, and hence determines a well-defined metric $\widehat{g}$ on $M$. By construction, $\widehat{g}$ is invariant under all local conformal transformations. The converse is immediate: if $f \in H_p$ is an isometry of some $\widehat{g}$ in the conformal class, with $f(p) = p$, then $(f^*g)_p = g_p$, and therefore $\mathrm{d}f_p \in \mathrm{O}(T_pM, g_p)$.

\item If $f \in H_p$ satisfies $d f_p = e^{\lambda} \1$, then for the $(1,3)$ Weyl tensor we have $ W_p = f^* W_p = e^{2\lambda} W_p$, so that $W_p = 0$ if $\lambda \neq 0$.
The statement then follows from conformal local homogeneity.
 
\item   The isotropy group $H_p$ is algebraic. This is a consequence of Gromov’s theory of rigid transformation groups, as explained in Appendix~\ref{Gromov's Theory}. Since $\ker \rho_p$ is a closed subgroup of $H_p$, it is also algebraic. In particular, it has finitely many connected components.
If the kernel is finite, then it must be trivial. Indeed, a finite group preserves a Riemannian metric; hence any element of the kernel is linearizable, that is, conjugate via the exponential map to its derivative. Therefore, such an element cannot have trivial derivative at $p$ unless it is itself trivial.
As a result, if $\rho_p$ is not injective, then the identity component $(\ker \rho_p)^0$ is non-trivial. Let $\phi^t \in (\ker \rho_p)^0$ be a non-trivial one-parameter subgroup such that $\rho_p(\phi^t)=\1$.
 Then, the set $\{ d_p \phi^t \}$ is relatively compact in $\CO(T_pM)$.
By \cite[Theorem 1.4]{FM2013}, which applies in dimension $\ge 3$, the flow is either inessential or the metric is conformally flat. 
However, since $\phi^t$ has trivial derivative, it must be essential. It follows that the metric is conformally flat.

 \item In the locally homogeneous case, for any $p\in M$ the isotropy group $H_p$ is identified with an algebraic subgroup $H$ of some $\GL(m,\R)$, and the derivative representation is identified with an algebraic homomorphism from $H$ to $\GL(T_p M) \simeq \GL(n+2,\R)$ (see Appendix~\ref{Gromov's Theory}). Then Lemma \ref{Lem: algebraic image} below implies that $\rho_p(H_p)$ contains the identity component of its Zariski closure in $\GL(n+2,\R)$, which is contained in $\CO(1,n+1)$ since the latter is Zariski closed in $\GL(n+2,\R)$.
\end{enumerate}
\end{proof}

 \begin{lemma}\label{Lem: algebraic image}
  Let $H$ be an algebraic subgroup of $\GL(m,\R)$, and $\rho: H \to \GL(n,\R)$ an algebraic homomorphism over $\R$. Then, $\rho(H)$ has finite index in its Zariski closure in $\GL(n,\R)$. In other words, their identity components coincide. 
 \end{lemma}
 
 \begin{proof}
Given a subgroup $G \subset \GL(n,\R)$, we denote by $G^0$ its identity component, by $\mathrm{Lie}(G)$ its Lie algebra, and by $G^\C$ its Zariski closure in $\GL(n,\C)$. If $G$ is an algebraic subgroup of $\GL(n,\R)$, then by \cite[Remark, p. 285]{EJ09} one has
 \begin{equation}\label{Eq-Idcomp}
     G^0 = G^\C (\R)^0,
 \end{equation}
where $G^\C(\R)=G^\C \cap \GL(n,\R)$.   Moreover, 
 \begin{equation}\label{Eq-Lie}
     \mathrm{Lie}(G)=\mathrm{Lie}(G^\C) \cap \gl(n,\R).
 \end{equation} 
 The homomorphism $\rho$ extends to an algebraic homomorphism over $\R$, $$\rho_\C: H^\C \to \GL(n,\C).$$
Its differential $d \rho$ which is an $\R$-linear map extends   by complexification to a $\C$-linear map. 
  Since $H$ is algebraic,  by (\ref{Eq-Idcomp}), its identity component satisfies    $H^0 = H^\C (\R)^0$. 
  Since the image of a complex algebraic group under an algebraic homomorphism is Zariski closed (see for instance \cite[Paragraph 7.4, Proposition B]{Hum} or \cite[Proposition 1.2.5 p. 11]{Perr}), it follows that $G:=\rho_\C(H^\C)$ is an algebraic subgroup  of $\GL(n,\C)$.  
  In fact, $G=\rho_\C(H^\C)=\rho(H)^\C$. 
 Let $R$ be the Zariski closure of $\rho(H)$ in $\GL(n,\R)$. Then $R$ is an algebraic group whose Zariski closure in $\GL(n,\C)$ is equal to the Zariski closure of $\rho(H)$ in $\GL(n,\C)$, hence to $G$. 
 Thus, by (\ref{Eq-Idcomp}), we have $R^0= G(\R)^0$.  Now, (\ref{Eq-Lie}) yields $\mathrm{Lie}(\rho(H))=\d\rho(\mathrm{Lie}(H))=\d\rho (\mathrm{Lie}(H^\C) \cap \gl(m,\R))$ and  $\mathrm{Lie}(R)=\mathrm{Lie}(G(\R))= \mathrm{Lie}(\rho_\C(H^\C)) \cap \mathfrak{gl}(n,\R)$. 
 Since $\rho$ is defined over $\R$, the map $d \rho_\C$
 commutes with complex conjugation, hence $\d\rho_\C(\mathrm{Lie}(H^\C)) \cap \gl(n,\R)=\d \rho (\mathrm{Lie}(H^\C) \cap \gl(m,\R))$. 
 It follows that $\mathrm{\rho(H)}=\mathrm{Lie}(R)$, and hence $\rho(H)$ and $R$ have the same identity component. 
 \end{proof}

The fact that the image of $H_p$ under $\rho_p$ contains the identity component of its Zariski closure in $ \CO(1,n+1)$ means that its Lie algebra is the Lie algebra of an algebraic subgroup of $\CO(1,n+1)$. This observation 
 is crucial in what follows.  
 The reason is that we can use the following fact from the theory of algebraic groups, which we formulate for subgroups of $\CO(1,n+1)$. For a proof, we refer to \cite[Theorem 15.3 p. 99]{Hum}, or \cite[Theorem 3.1.6 p. 30]{Perr}.

\begin{lemma}[Jordan decomposition in algebraic groups]\label{Jordan-lem}
Let $\H \subset \CO(1,n+1)$ be an algebraic group, and let $\h$ be its Lie algebra. 
\begin{enumerate}
\item Let $B \in \H$.
Then there exist (unique) commuting elements $B_s, B_u \in \H$ such that $B = B_s B_u$, where $B_s$ is semisimple (i.e.~diagonalizable over $\C$) and $B_u$ is unipotent. 
In particular, $B_u \in \H \cap \O(1,n+1)$. (This is called the Jordan decomposition of $B$ in $\H$). 
\item Let $X \in \h$. Then there exist unique elements $X_s, X_u \in \h$ such that $X=X_s + X_u$, and $e^{tX} = e^{t X_s} e^{t X_u}$ is the Jordan decomposition of $e^{tX}$ in $\H$. 
\end{enumerate}

Moreover, there exist commuting elements $B_h, B_e \in \H$ such that $B_s = B_h B_e$, where $B_h$ is hyperbolic (i.e. diagonalizable, with only real eigenvalues) and $B_e$ is elliptic (i.e.~all eigenvalues lie on $\S^1$).
In particular, $B_e \in \H \cap \O(1,n+1)$.
\end{lemma}

Applying this to the image $\H_p$ of the linear isotropy representation yields the following result.
 
\begin{corollary} \label{essential-cor}
Assume that $(M, g)$ is weakly essential  and locally conformally homogeneous. Then 
\begin{enumerate}
\item The  image $\H_p := \rho_p(H_p) \subset \CO(T_pM, g_p)$ of $H_p$ contains the identity component of its Zariski closure in $\CO(T_p M) \simeq \CO(1,n+1)$, and it is not contained in $\O(T_pM, g_p)$.

\item ${\bf H}_p$ has finitely many connected components. In particular, its Lie algebra $\h_p \subset \co(T_pM)=\R \+  \so(T_pM, g_p) \subset \gl(T_pM)$ is non-trivial, and it is not contained in $\so(T_p M,g_p)$.

\item  ${\bf H}_p$ contains a  one-parameter subgroup $B^t \in  \CO(T_pM, g_p)$ of the form $B^t= \e^{t \alpha} \e^{t A}$, where $\alpha >   0$, and 
 $A \in \so(T_pM, g_p)$ is an infinitesimal (nontrivial) boost, i.e.~$A$ is hyperbolic with eigenvalues $\pm 1$ and  $0$. Up to permutation of the basis, $A=\diag(1,-1,0,\ldots,0)$,
 so that $B^t= e^{t \alpha} \diag{(e^t, e^{-t}, 1, \ldots, 1)}$.
\end{enumerate}
\end{corollary}
 
\begin{proof} 
\begin{enumerate}
\item This is a consequence of Proposition \ref{Prop. isotropy}. 
\item Since $H_p$ is an algebraic group (see Appendix \ref{Gromov's Theory}), it has finitely many connected components; this is a general property of algebraic groups (see, for example, \cite[Proposition 7.3, p. 53]{Hum} or \cite[Proposition 1.2.1, p. 10]{Perr}). Hence, the same holds for its image  $\H_p=\rho(H_p)$. Moreover, by (1), $\H_p$ is not contained in $\O(T_pM, g_p)$, and therefore is not a finite group. Thus, its Lie algebra is non-trivial.
Now, let $f \in \H_p$ be an element that does not belong to $\O(T_pM, g_p)$. Since $\H_p$ has finitely many connected components, some iterate $f^k$ lies in the identity component $\H_p^0$ of $\H_p$, and still satisfies $f^k \notin \O(T_pM, g_p)$. Hence, the identity component $\H_p^0$ is not contained in $\O(T_p M, g_p)$, 
which implies that its Lie algebra $\h_p$ is not contained in $\so(T_pM, g_p)$.

\item By (2),  $\h_p$ is not contained in $\so(T_pM, g_p)$. Hence, there exists a one-parameter subgroup $B^t$ that is not contained in $\O(T_pM, g_p)$. The given form  of $B^t$ is then a consequence of Lemma \ref{Jordan-lem}.
\end{enumerate}
\end{proof}

\begin{remark} From a dynamical point of view, our conformal one-parameter group  fixing $p$ and generated by $\alpha \1 +A$ is expanding if $\alpha >1$, 
has a central manifold of dimension 1 if $\alpha = 1$, and is hyperbolic with a stable manifold of dimension 1 if $\alpha <1$. 
We will however treat all these cases uniformly, by showing that they lead to a plane wave structure or conformal flatness. 
\end{remark}
 
\subsection{Reduction to parabolic isotropy} 
Recall that $\so(1,n+1)$ is a $|1|$-graded Lie algebra when decomposed into the eigenspaces of $\mathsf{ad}_A$, where $A$ is the grading involution $A=\diag{(1,-1,0,\ldots, 0)}$,
\begin{equation}
\label{grading}
\so(1,n+1)= \s^{-1}\+\s^0\+\s^{1}, \quad \text{ with }\s^0=\co(n)=\R A\+\so(n),\end{equation}
and with  $(\s^1)^*\simeq \s^{-1}$ via the Killing form of $\so(1,n+1)$, both isomorphic as $\co(n)$-modules to $\R^{n}$.
We also split $\R^{1,n+1}$ into the eigenspaces of $A$ as
\[\R^{1,n+1}=V^{-1}\+V^0\+V^{1},\]
so that $\s^{\mu}(V^{\nu})\subset V^{\mu+\nu}$, where we declare $V^\mu=\{0\}$ if $\mu\not\in \{-1,0,1\}$.
 The maximal parabolic subalgebras \[\p^\pm=\s^0\ltimes \s^\pm\]  are the stabilizers of the null lines $V^{\pm1}$, respectively.
On $\s^\mu$, $A^t:=\mathrm{e}^{tA} = \diag{(\e^t, \e^{-t}, 1, \ldots, 1)}$ acts as 
\[
\mathsf{Ad}_{A^t}: \so(1,n+1)\ni X \longmapsto A^tXA^{-t}\in \so(1,n+1),\]
with eigenvalues $\e^{\pm t}$ and $1$. 
Now we have the following.

\begin{lemma}
\label{para-lem}
Let $\H \subset \CO(1,n+1)$ be an algebraic subgroup with Lie algebra $\h$ and assume that $\h$ is not contained in $\so(1,n+1)$. Then either $\h$ contains the identity or 
 $\h$ stabilizes a null line in $\R^{1,n+1}$.
\end{lemma}
\begin{proof}
Let $W\subset \R^{1,n+1}$ be the orthogonal complement of the largest positive definite subspace of $\R^{1,n+1}$ that is invariant under $\h$. 
If $W$ is $1$-dimensional, then $\h$ is contained in $\R\1\+\so(n+1)$. Since $\H$ is assumed to be algebraic, Lemma~\ref{Jordan-lem} implies that $\1\in \h$. Otherwise,
$W$ is isomorphic to a Minkowski space $\R^{1,k}$, with $1\le k\le n+1$, that does not admit any non-degenerate $\h$-invariant proper subspace. Let $\check{\h}\subset \co(W)\simeq\co(1,k)$ denote the restriction of $\h$ to $W$ and $\check{\h}_0$ its projection onto $\so(W)\simeq\so(1,k)$.  Since $\check{\h}_0$ has the same invariant subspaces as $\check{\h}$, we have that either $\check{\h}_0$ admits an invariant null line in $W$, in which case also $\check{\h}$ and hence $\h$ admit an invariant null line and we are done, or $\check{\h}_0$ is irreducible and therefore, by \cite{olmos-discala01}, equal to $ \so(1,k)$. In the latter case, $\check{\h}$ contains the grading element $A$ of $\so(1,k)$. Indeed, since $\check{\h}_0=\so(1,k)$, there are $v_\pm\in \s^\pm\subset \so(1,k)$ with $[v_+,v_-]=A$ and $a_\pm\in \R$ such that $w_\pm=a_\pm\1+v_\pm\in \check{\h}$. Then  and $[w_+,w_-]=A\in \check{\h}$. This implies that 
$A+X\in \h$, for an $X\in \so(n+1-k)$ commuting with $A$, so that $\exp(A)\exp(X)\in \H$. Then by Lemma~\ref{Jordan-lem}, $\exp(A)\in \H$, and so $A\in \h$.

Finally, since $\h\not\subset\so(1,n)$, it contains 
 $\alpha \1+C$ with $C\in \so(1,n+1)$, but again by the assumption of $\H$ being algebraic and Lemma \ref{Jordan-lem}, $C$ must be a multiple of the grading element $A$, which we have shown to be contained in $\h$. Hence, $\1\in\h$.
\end{proof}
\begin{corollary} 
\label{para-cor}
Assume that $(M, g)$ is weakly essential and locally conformally homogeneous with isotropy $\H_p$ at $p\in M$. Then either $(M,g)$ is conformally flat or  the Lie algebra $\h_p$ of $\H_p$ is contained in the stabilizer in $\co(T_pM)$ of  a null line, i.e. in   $\R\1\+\p^\pm$.
\end{corollary} 

\begin{proof}
By Proposition \ref{Prop. isotropy},   the Lie algebra  of $\H_p$ is the Lie algebra of an algebraic subgroup of $\CO(1,n+1)$. Moreover,  by Corollary \ref{essential-cor}, it is not contained 
in $\so(1,n+1)$. 
Hence Lemma \ref{para-lem} applies,  and we conclude that either $\h_p$ contains the identity or it stabilizes a null line. In the former case, Proposition \ref{Prop. isotropy} implies that $(M,g)$ is conformally flat.
\end{proof}

\subsection{The simply connected case}
 Observe that $(M,g)$ and its universal cover have the same  local conformal transformations. Moreover, 
 in the locally conformally homogeneous setting, the characterization of weak essentiality  in Proposition \ref{Prop. isotropy} implies that $(M,g)$ is weakly essential if and only if its universal cover is weakly essential. 
Hence, when assuming weak essentiality and local conformal homogeneity, we may assume that $M$ is simply connected. With this assumption, we have the following properties.

\begin{proposition}\label{Prop. M 1-connected} 
Let $(M,g)$ be a locally conformally homogeneous pseudo-Riemannian manifold, with  $M$ simply connected. Let $\H_p$ be the image of the isotropy representation as before. Then 

\begin{enumerate}
\item  Any conformal vector field defined on an open subset of $M$ extends to a globally defined conformal vector field on $M$.
 
\item  Let $\g$ be the Lie algebra of global conformal vector fields on $M$. Then $\g$ acts transitively on $M$; that is, for any $p \in M$, the evaluation map $X \in \g \mapsto X(p) \in T_pM$
is surjective.
\item  Assume $M$ is non-conformally flat.
Let $\g_p = \{ X \in \g \mid X(p) = 0 \}$ be the stabilizer algebra of $p$.
It is naturally identified with $\h_p$ (the Lie algebra of $\H_p$) via the map $ \g_p\ni X \mapsto 
\nabla X|_p\subset \co(T_pM)$, where $\nabla$ is the Levi-Civita connection of $g$.
\end{enumerate} 

\end{proposition}
 
\begin{proof}
\begin{enumerate}
\item This is shown in 
\cite[Lemma 3]{LedgerObata70} under the assumption that $(M,g)$ is analytic. 
Because of local conformal homogeneity, this assumption is satisfied.

\item See Appendix \ref{Gromov's Theory}, Paragraph `The pseudogroup action and local Killing fields'. 
\item By Appendix \ref{Gromov's Theory}, Paragraph `Local isotropy via the pseudogroup action', $\g_p$ is identified with the Lie algebra of $H_p$, which, by Proposition \ref{Prop. isotropy}, item (3), is in turn identified with $\h_p$, the Lie algebra of $\H_p$.
\end{enumerate}
\end{proof}
This proposition allows us, in what follows, and up to passing to the universal cover, to work with the (essential and transitive) Lie algebra $\g$ of global conformal vector fields.\bigskip

\subsection{The $B^t$-adjoint-action on $\g$}\label{B-sec}
We assume that $(M, g)$ is weakly essential.
 By Corollary \ref{essential-cor}, $\H_p$ contains a one-parameter subgroup $B^t= \e^{t \alpha} \e^{t A}$, where $\alpha >0$ and $A$ is the grading element of $\so(1,n+1)$.
We want to understand the adjoint action of $B^t= \e^{t \alpha} \e^{t A}$ on the Lie algebra $\g$ of global conformal vector fields of $(M,g)$.  
We further assume  that the isotropy representation $\rho_p$ is injective, which, by Proposition \ref{Prop. isotropy}, holds when the manifold is not conformally flat. Under this assumption, the isotropy algebra $\g_p$ identifies with its image $\h_p$ (see the third point of Proposition \ref{Prop. M 1-connected}).  
For simplicity of notation, we set $\h := \h_p$. 
Then, as a vector space, $\g$ is isomorphic to  $\g \cong \h \oplus (\g/\h)$, 
where $\g/\h$ is naturally identified with $T_pM$.\medskip

As for the $B^t$-adjoint action, it can be described as follows:
\begin{itemize}
  
\item Since  $\mathsf{Ad}_{B^t}$ preserves $\h$, it acts on 
 $\g/\h$. Its action there coincides with that of   $B^t$  on $T_pM$, that is,\
   \[B^t= e^{t \alpha} \diag{(e^t, e^{-t}, 1, \ldots, 1)} = \diag{(e^{(\alpha+1)t}, e^{(\alpha-1)t}, e^\alpha, \ldots, e^\alpha)}
.\]

\item On $\h$, $B^t$ can be seen as a one-parameter subgroup of $\CO(T_pM)=\R^*\times \O(T_pM, g_p)$, and $\h$ is a Lie subalgebra of $\co(T_pM)=\R \+ \so(T_pM, g_p)$. Then the  $\mathsf{Ad}_{B^t}$-action
 on $\h$ is simply the restriction of its $\mathsf{Ad}$-action on $\co(T_pM)$,
\[\Ad_{B^t} : \co(T_pM) \rightarrow \co(T_pM), \quad X \mapsto B^t X B^{-t}.\]
 Since $B^t$ acting on $T_pM$ is diagonalizable, the adjoint action $\mathsf{Ad}_{B^t}$  on $\co(T_pM)$ is also diagonalizable, with eigenvalues  $e^{\pm t}$ and $1$. 
Therefore,  the  $\mathsf{Ad}_{B^t}$-action  on $\h$ is diagonalizable with with eigenvalues forming a subset of those of $\Ad_{B^t}$ on $\co(T_pM)$.
\end{itemize}
 
Observe that the adjoint action $\Ad_{B^t}$ on $\g$ is not necessarily diagonalizable, since it may fail to preserve a subspace complementary to $\h$ and isomorphic to $\g/\h \cong T_pM$. However, in this situation we have the following.

\begin{lemma}\label{Aut-prop} 
There exists a one-parameter group of diagonalizable automorphisms  $\B_t$ of $\g$ 
whose action coincides with that of $\Ad_{B^t}$ on both $\h$ and $\g/\h$ (that is, with an invariant subspace complementary to $\h$). 
\end{lemma}
 
\begin{proof} 
The automorphism group of a Lie algebra $\g$ is an algebraic subgroup of $\GL(\g)$ (since it consists of those elements of $\GL(\g)$  preserving the bracket). 
Hence, when we decompose $\Ad_{B^t} = S_t U_t$ into its semisimple and unipotent parts, 
both $S_t$ and $U_t$ are automorphisms of $\g$. 
Then, $S_t$ is the required diagonal one-parameter group of automorphisms, having the same eigenvalues as $\Ad_{B^t}$ on $\h$ and $\g/\h$.
\end{proof}

\subsection{Canonical $\g$-invariant lightlike distributions}
\label{N-sec} 
In Corollary \ref{para-cor}, we have seen that if $(M,g)$ is weakly essential and non conformally flat, the image  of the isotropy algebra at $p$ under $\rho_p$ is contained in the stabilizer of a null line in $T_pM$, and hence leaves this null line invariant. First we will describe this situation more generally for any invariant subspace.\medskip

Let $\g$ be the  Lie algebra  of conformal vector fields of $(M,g)$.   We have an embedding of $\g$ into the vector fields of $M$,  $\g\hookrightarrow\mathfrak{X}(M)$. Let $\mathbf{G}$ be the connected and simply connected Lie group with Lie algebra $\g$. There exists a local action of $\mathbf{G}$ on $M$ such that $X \in \g$ coincides with the the fundamental vector field $q \in M \mapsto \widetilde{X}(q):=\tfrac{\d}{\d t} (\exp(tX)(q))|_{t=0}$  associated to $X$ (see  \cite[Chapter II, Theorem XI and its corollary]{Pal56}).

If $q\in M$, then by $\g|_q$ we denote the image of the evaluation map  at $p$ by
\[\g|_q:=\{ {X}(q)\mid X\in \g\}.\]
For $\psi\in \mathbf G$ acting locally on $M$, with $\psi^*g=f^2 g$, we have
\begin{equation}\label{push-ad}
 \psi_*\widetilde{X}=\widetilde{\Ad_{\psi}(X)}.\end{equation}
Hence we relate the metric at $\psi(p)$ to the metric at $p$ by
$$ g (\widetilde{X},\widetilde{Y})_{|\psi(p)}= f^{-2}(\psi(p))  g_p(\widetilde{\Ad_{\psi^{-1}}(X)}_p,\widetilde{\Ad_{\psi^{-1}}(Y)}_p) .$$

\begin{proposition}\label{distributionprop}
Let $\g$ be the algebra of conformal vector fields of $M$, and assume that $\g$ acts transitively on $M$.  
Fix a point 
$p\in M$, and let $\g_p$ be the isotropy subalgebra at $p$, with image $\h_p$ under the isotropy representation $\rho_p$.  
Let $V_p$ be a subspace of $T_pM$ that is invariant under $\h_p \subset \co(T_pM)$. Then
\begin{enumerate}
\item There is a vector distribution $V$ with $V_p$ as fiber at $p$ that is invariant under $\g$, i.e. $\d\psi_p(V_p)=V_{\psi(p)}$ for all local conformal transformations $\psi \in \mathbf{G}$ acting locally on $M$.
\item If there is a Lie algebra $\l\subset \g$ with stabilizer algebra $\l_p$ such that the isomorphism $T_pM\simeq \g/\g_p$ restricts to an isomorphism $V_p\simeq \l/\l_p$, then $V$ is involutive with fiber at $\psi(p)$ spanned by $\Ad_\psi (\l)|_{\psi(p)}$, for $\psi\in \mathbf{G}$ acting locally on $M$. 
In particular, if $\l$ is an ideal of $\g$, then $V_q$ is spanned by $\l|_q$ for each $q\in M$.
\item If $\mathbf{L}$ is the Lie group corresponding to $\l$, then  the leaves through $q=\psi(p)$ are given 
by the orbits of $q$ under $\psi \mathbf L \psi^{-1}$. 
In particular, if $\l$ is an ideal, then the integral manifold through $q$ is the orbit of $q$ under $\mathbf{L}$.
\end{enumerate}
\end{proposition}

\begin{proof}
\begin{enumerate}
\item By the $\g$-transitivity, one can transport $V_p$ to any point: for $q=\psi(p)$,  with $\psi\in \mathbf G$, we define $V_q:=\d\psi_p (V_p)$. Since $\h_p$ leaves $V_p$ invariant, this is well-defined. This defines a vector distribution $V$  that is invariant under the action of $\g$.
\item It follows from equation (\ref{push-ad}) that
 $V_q=\d\psi_p(V_p)=\d\psi_p(\l|_p) =(\psi_*\l)|_q=\Ad_\psi(\l)|_q$. 
 Since $\l$ is a Lie algebra, $V$ is involutive.
\item  Let $L$ be the orbit of $p$ under the (local) action of the connected Lie group $\mathbf L$ with Lie algebra $\l$.   Then
$T_p L = \l|_p=\{\widetilde{X}(p) \mid V \in \l\}$. 
By the condition on $\l$, we also have that $T_pL=V_p$, so that $L$ is tangent at $p$ to $N_p$. Moreover, $L$ is an integral manifold of $V$ through $p$. 
Indeed, for any $q=\psi(p)\in L$ with $\psi\in \mathbf{L}$, we have 
$V_q=\d \psi_p(V_p) = \d \psi_p(T_pL)=T_qL$, 
i.e. $N$ is tangent to $L$ everywhere on $L$.  
Finally, by local transitivity of $\mathbf{G}$, $V$ is tangent to images of $L$ under local conformal transformations.
In particular, if $q=\psi (p)$, then the integral manifold of $V$ though $q$ is the orbit under the conjugated group $\psi \mathbf{L} \psi^{-1}$. 
\end{enumerate}  
\end{proof}

Now we return to our situation.
 Let  $B=\diag{(\alpha+1,\alpha-1, \alpha,\dots, \alpha)}\in \h_p$, and denote by $V^\mu$ the eigenspace of $B$ in $T_pM$ corresponding to the eigenvalue $\mu$. Then
$$T_pM= V^{\alpha+1}\+ V^{\alpha-1}\+V^\alpha,$$
where $ V^{\alpha\pm1}$ are null lines and $V^\alpha $ is a Euclidean space. 
Note that 
\[\s^\mu(V^\nu) = V^{\mu+\nu},\]
where $\s^\mu$ is defined as in (\ref{grading}). 
In particular, the parabolic subalgebras 
\[\p^\pm=\s^0\ltimes \s^\pm=\mathfrak{stab}_{\so(1,n+1)}(V^{\alpha\pm 1})\]  are precisely the stabilizers of the null lines $V^{\alpha\pm1}$, respectively.\medskip

By Corollary \ref{para-cor}, if $(M,g)$ is weakly essential, then it is either conformally flat or  $\h_p$  is contained in the stabilizer in $\co(1,n+1)$ of a null line. Thus,  the first point in Proposition \ref{distributionprop} yields the following.
\begin{corollary}
Let $(M,g)$ be a Lorentzian manifold with an essential and transitive
Lie algebra $\g$ of conformal  vector fields. Then either $(M,g)$ is  conformally flat, or 
it admits 
a null line distribution $N$, and hence a null hyperplane distribution $N^\perp$, invariant under the action of $\g$, with $N^\perp_p=V^\alpha\+V^{\alpha\pm 1}$. 
\end{corollary}

\section{Towards a Heisenberg algebra of Killing vector fields}\label{Section: Heis algebra}
From now on, we assume that the isotropy representation $\rho_p$ is injective, so that the isotropy algebra $\g_p$ identifies with its image $\h:=\h_p$ under $\rho_p$, and $\g$ is, as a vector space, isomorphic to $\g \cong \h \oplus (\g/\h)$.  We also assume that  $\h$  is contained in either $\R\+\p^+$  or  $\R\+\p^-$. 
For $\h\subset \R\+ \p^\pm$,   $\ad_B$ acts on $\h$ with eigenvalues $0$ and  $\pm 1$, respectively. 
As explained in Section \ref{B-sec}, the induced action of $\ad_B$ on $\g/\h$ coincides with the action of $B$ on the tangent space, i.e. with eigenvalues $(1\pm\alpha)$ and $\alpha$. Consequently, the set of generalized eigenvalues of $\ad_B$ on $\g$ are
$$\sigma_B := \{\pm1, 0, \alpha-1,\alpha, \alpha+1\}.$$
Now, let $\B=B_s\in \mathfrak{aut}(\g)$ be the semisimple part of $\ad_B$, as given by Lemma \ref{Aut-prop}. Then $\B$ is diagonalizable, with eigenvalues contained in $\sigma_B$.
Denote by $\g^\mu$ the corresponding eigenspaces, so that
$$\g=\sum_{\mu\in \sigma_B} \g^\mu.$$ 
Since $\B$ is a derivation of $\g$, the eigenspace decomposition satisfies
\begin{eqnarray}\label{roots}
\left[ \g^\mu,\g^\nu \right]&\subset& \g^{\mu+\nu},
\end{eqnarray}
where, by convention, $\g^\mu=\{0\}$ if $\mu\not\in\sigma_B$. 
We obtain the decomposition
$$\g=(\g^0\+\g^{\pm 1}) + (\g^{\alpha-1}\+\g^\alpha\+\g^{\alpha+1}), \quad \h\subset \g^0\+\g^{\pm 1},
$$ 
where the sum $+$ in the middle need not be direct.
 \medskip

Define  $\h^{00}:=\g^0 \cap \h \cap \so(1,n+1) \subset \so(n)$. We will also assume that $\g^0=\R B \oplus \h^{00}$; otherwise, $\h$ would contain the homothetic diagonal one-parameter group $\diag{(e^{\alpha t}, \ldots, e^{\alpha t})}$, and the metric would be conformally flat. \medskip

Notation: Given a subalgebra $\k \subset \g$, we  write $\k_p:=\k \cap \h_p$, where $\h_p=\h$ denotes the stabilizer algebra at $p\in M$. \medskip

Since, by assumption, $\alpha \neq 0$, we may further assume that $\alpha > 0$ by replacing $B$ with $-B$ if necessary. With the sign of $\alpha $ fixed, we have to consider both cases $\h\subset \g^0\+\g^+$ and $\h\subset \g^0\+\g^-$.
The goal of this section is to prove the following proposition.

\begin{proposition}\label{Prop. existence of foliation}
Let $(M,g)$ be a Lorentzian manifold of dimension $n+2$, with a transitive Lie algebra $\g$ of conformal vector fields, and  stabilizer algebra $\h$  at $p\in M$.
Assume that 
$\h \subset \g^0 \oplus \g^{\pm 1}$ and $\g^0=\R B \oplus \h^{00}$ with $B=\mathrm{diag}(\alpha+1,\alpha-1,\alpha,\ldots, \alpha)$. 
Then there is a metric $\hat g$ in the conformal class of $g$ and 
an algebra  of Killing vector fields $\g_1$ for $\hat g$ that acts locally transitively at $p$. Moreover, 
\begin{enumerate}
    \item If  $\alpha \notin \left\{\frac{1}{2}, 1, 2\right\}$, then  $(M,g)$ is conformally flat.
    
    \item If $\alpha \in \{\frac{1}{2}, 1, 2\}$ and $\h \subset \g^0 \oplus \g^1$, then there exists a codimension-one lightlike foliation $\F$ on $M$, such that the ideal $\l \subset \g$  of conformal vector fields preserving the leaves of $\F$ individually, satisfies:
    \begin{itemize}
    \item[(a)] there exists a subalgebra $\f \subset \l$ of Heisenberg type, containing $\g^{\alpha+1}$ in its center, and acting locally transitively on the leaves of $\F$,
    \item[(b)] the null line bundle $\F^\perp$ is defined by the one-dimensional subspace $\g^{\alpha+1} \subset \f$, and 
    \item[(c)] $\f$ is contained in $\g_1$, i.e.~it acts by Killing fields on the metric $\hat{g}$ in the conformal class.
    \end{itemize}
    \item If $\alpha \in \{\frac{1}{2}, 1, 2\}$ and $\h \subset \g^0 \oplus \g^{-1}$, such a foliation exists with the same properties in a neighborhood of $p$. 
\end{enumerate}
\end{proposition}

\begin{lemma}\label{Lem. CF case}
If $\alpha \notin \{\frac{1}{2}, 1, 2\}$ and $\h \subset \g^0 \oplus \g^{\pm 1}$, then $\m:=\g^{\alpha-1} \oplus \g^\alpha \oplus \g^{\alpha+1}$  is an abelian subalgebra acting locally simply transitively on $M$. Moreover, there is a metric in the conformal class for which $\m$ acts by Killing fields. In particular, the metric is conformally flat.
\end{lemma}
\begin{proof}
That $\m$ is abelian is immediate from (\ref{roots}) and then the Lemma follows from $\m_p=\m\cap \h=\{0\}$ and the first statement in Proposition~\ref{Prop. isotropy}.     
\end{proof}
From now on, we assume that $\alpha \in \{\frac{1}{2}, 1, 2\}$. These cases are more involved.

\subsection{Case $\h \subset \g^0 \oplus \g^1$}
When $\mathfrak{h} \subset \mathfrak{g}^0 \oplus \mathfrak{g}^{1}$, the canonical distribution $N^\perp$ defined in  Section~\ref{N-sec} is integrable, and defines a codimension-one lightlike foliation $\F^+$ on $M$, invariant under the action of $\g$. Indeed, define
$$\n^{+}:=\g^\alpha \oplus \g^{\alpha + 1}.$$
Let $\n$ denote the algebra generated by $\n^+$. Since $[\n^+,\n^+]\subset \n^+ + \h$, we have $\n|_p = \n^+ |_p = V^\alpha \+ V^{\alpha +1}=N_p^\perp$, hence $N_p^\perp \simeq \n/\n_p$. The claim then follows from Proposition \ref{distributionprop}.  \bigskip

Let $\xi$ be a local curve through $p$ transverse to $\F^+$.  Because $\F^+$ is $\g$-invariant, $\g$ induces an action on $\xi$. Let $\mathfrak{X}(\xi)$ denote the (infinite dimensional) Lie algebra of vector fields on $\xi$. Then we have a Lie algebra representation 
\begin{equation}\label{Eq. rho}
 \rho: \g \to \mathfrak{X}(\xi),    
\end{equation}
whose kernel consists precisely of the conformal vector fields that are tangent to $\F^+$ in some neighborhood of $p$. \medskip

By studying the action of $\g$ transversely to $\F^+$, that is, on $\xi$, we will show that the subalgebra generated by $\n^+$ preserves each leaf of $\F^+$ individually, and that this action is locally transitive on every leaf.  
In fact, in this situation, this algebra will be of Heisenberg type, and will thus provide the desired foliation $\F:=\F^+$ described in Proposition \ref{Prop. existence of foliation}.

\begin{lemma}\label{Lem. existence of foliation}
 In a neighborhood of $p$, $\n^+$ acts by preserving all the leaves of $\F^+$ individually (equivalently, the subspaces $\n^+|_q \subset T_q M$ are all tangent to $\F^+$ in that neighborhood). Moreover, $\g^{\alpha+1}$ is $1$-dimensional and $\g$-invariant, hence, generates the null line bundle ${\F^+}^\perp$. 
\end{lemma}
\begin{proof} 
\noindent \textbf{Case} $\boldsymbol{\alpha \in \{1,2\}}$.   In this case, $\mathfrak{n}^+$ is an ideal of $\mathfrak{g}$. Hence, by (3) in Proposition~\ref{Prop. existence of foliation}, the $\mathfrak{g}$–invariant foliation tangent to $\mathfrak{n}^+|_p$ at $p$, namely $\F^+$, is everywhere tangent to the distribution $q \mapsto \n^+|_q$. Moreover, $\g^{\alpha+1}$ is $\g$-invariant, hence everywhere tangent to $\F^+$. \medskip

\noindent \textbf{Case} $\boldsymbol{\alpha = \frac{1}{2}}$. 
In this case, $\g^0 = \R B \oplus \h^{00}$, and $\h=\g^0 \oplus \g^1=\R B \oplus \h^{00} \oplus \g^1$. 
We will show that $\rho(\g^\alpha)=0$. 
Let $U$ be a generator of $\g^{\alpha-1}$, and let $X \in \g^{\alpha}$. Then 
\[[B,U]=-\tfrac{1}{2} U, \quad [B,X]=\tfrac{1}{2} X, \quad [X,U] \subset \R B \oplus \h^{00}.\] Since $U$ is transverse to $\F$ in a neighborhood of $p$, its image under $\rho$ is non-zero. Moreover, $[\rho(B),\rho(U)]=-\frac{1}{2} \rho(U)$, so $\rho(B)$ is also non-zero  and not collinear with $\rho(U)$. The following observation will be used several times:  given $\varphi \in \h^{00}$, the element $\rho(\varphi)$ cannot equal  $a \rho(B)$ for any $a \in \R \smallsetminus \{0\}$, since $\rho(B)$ acts non-trivially on $\g^{\alpha -1}$,  whereas $\h^{00}$ acts trivially on it. 
A finite dimensional Lie algebra of vector fields on $\R$ has maximal dimension $3$. One can check that if $\rho(X) \neq 0$ for some $X \in \g^\alpha$,  then $\rho(\g^\alpha) \subset \R \rho(X)$. Hence, $\rho(\g^\alpha)$ has dimension at most one. 
Thus, with $\mathrm{dim}(\g^\alpha)=n$, there exist at least $(n-1)$ linearly independent conformal vector fields $e_1, \dots, e_{n-1} \in \g^{\alpha}$ such that $\rho(e_i)=0$ for all $i=1,\dots,n-1$.
Since $[\g^\alpha, \g^{\alpha -1}] \subset \h^0=\R B \oplus \h^{00}$, this, together with the observation above, implies that $[e_i,U] \in \h^{00}$, for all $i=1,\dots,n-1$. 
Let $V$ be a generator of $\g^{\alpha +1}$, and set $W:=[V,U] \in \h^1$. The Jacobi identity gives $[W,X]=\frac{3}{2} \lambda_X V$, where  $\lambda_X \in \R$ is defined by the condition $[X,U]-\lambda_X B \in \h^{00}$. In particular, $[W,e_i]=0$ for all $i=1,\dots,n-1$. 
Hence, either $W=0$, or $[W,\g^{\alpha-1}]= \R e_{n}$, where $e_{n} \in \g^\alpha$ is nonzero and satisfies $[W,e_{n}] \neq 0$. Assume we are in the latter case. Then the Jacobi identity applied to the triple $(e_{n}, U, W)$ yields a contradiction, which implies that $W=0$. Consequently, $[e_{n}, U] \in \h^{00}$.  
If $[\rho(e_{n}), \rho(U)] = 0$, then $\rho(e_{n}) = 0$, and the argument is complete. Indeed, since $\rho(U)$ acts freely on $\xi$ near $\xi(p)$ and $\rho(e_{n})$ vanishes at $\xi(p)$, the commutation relation implies that $\rho(e_{n})$ vanishes in a neighborhood of $\xi(p)$; equivalently, $\rho(e_{n})$ is tangent to $\F^+$ near $p$.
We now assume that $[\rho(e_{n}), \rho(U)] \neq 0$. In particular, $\rho(e_{n}) \neq 0$. Since $\rho(\mathfrak{g})$ is a Lie algebra of vector fields of maximal dimension $3$, we  must have $[\rho(e_{n}),\rho(U)]=a \rho(B)$, for some $a \in \R \smallsetminus \{0\}$. 
This, however, contradicts the observation above. Therefore, $\rho(e_{n})=0$, and hence $\rho(\g^\alpha)=0$. Finally, $\g^\alpha$ is tangent to $\F^+$ in a neighborhood of $p$. Moreover, we have also shown that 
\begin{equation}\label{Eq. bracket}
 [\g^{\alpha-1},\g^\alpha] \subset \h^{00}.   
\end{equation} 
It remains to show that $\g^{\alpha+1}$ is $\g$-invariant. All we need to show is that $[\g^{\alpha-1},\g^{\alpha+1}]=0$. 
Since $[\g^{\alpha-1},\g^\alpha] \subset \h^{00}$, the Jacobi identity applied to the triple $(\g^{\alpha-1},\g^\alpha,\g^{\alpha+1})$ yields 
\begin{equation}\label{Eq. bracket_1}
[\g^{\alpha-1},\g^{\alpha+1}]=0.   
\end{equation}  
\end{proof}

\begin{corollary}\label{Cor. ideal defining F}
The foliation  $\F^+$ is  defined by the distribution $E$ given by $q \mapsto E(q) = \l|_q$, where $\mathfrak{l}$ is the ideal of $\mathfrak{g}$ generated by $\mathfrak{n}^+$. 
\end{corollary}
\begin{proof}
Since $\F^+$ is $\mathfrak{g}$–invariant, the subalgebra $\mathfrak{k} \subset \mathfrak{g}$ consisting of conformal vector fields tangent to $\F^+$ in a neighborhood of $p$ is an ideal of $\mathfrak{g}$. We have shown that the conformal vector fields in $\mathfrak{n}^+$ are tangent to $\F^+$ in a neighborhood of $p$; equivalently, $\mathfrak{n}^+ \subset \mathfrak{k}$. It follows that the ideal $\mathfrak{l}$ generated by $\mathfrak{n}^+$ is contained in $\mathfrak{k}$, and hence that $\F^+$ is tangent to the distribution $E$ in a neighborhood of $p$. Consequently, the $\mathfrak{g}$–invariant foliation tangent to $\mathfrak{l}|_p$ at $p$, which coincides with $\F^+$, is everywhere tangent to the distribution $E$. 
\end{proof}

\begin{lemma}\label{Lem. (1) Heis subalgebra}
In $\g$, $\n^+$ is a  subalgebra   
of Heisenberg type, acting locally transitively on each leaf of $\F^+$, and containing  $\g^{\alpha+1}$ in its center. 
Moreover,  there exists  a metric in the conformal class,  admitting an algebra of 
Killing vector fields that acts locally transitively and contains $\n^+$.
Finally, the null line bundle ${\F^+}^{\perp}$ is (everywhere) generated by $\g^{\alpha+1}$.
\end{lemma}

\begin{proof}
\noindent \textbf{Case}  $\boldsymbol{\alpha = 2}$. In this case, $\n^+$ is an abelian ideal acting transitively on each leaf of $\F^+$. Let $\m :=  \n^+ \oplus \g^{\alpha - 1}$. 
Then $\m$ is an ideal of $\g$ such that $\m_p = \h \cap \s^+ \subset \so(T_pM)$, and $\m|_p$ spans $T_p M$. Hence, $\m$ acts transitively on $M$, and there exists a metric in the conformal class for which $\m$ is an algebra of Killing fields.  

\noindent \textbf{Case} $\boldsymbol{\alpha = 1}$. In this case, $\mathfrak{n}^+$ is a $2$-step nilpotent ideal acting transitively on each leaf of $\F^+$. Let $\mathfrak{c}$ be a complement of $\mathfrak{h} \cap \mathfrak{g}^{\alpha - 1}$ in $\mathfrak{g}^{\alpha - 1}$, and define
$
\m :=  \n^+ \oplus \c
$. Since $\c$ is of  
dimension~$1$,
 $\mathfrak{m}$ is a  subalgebra of $\mathfrak{g}$. By construction,  $\m_p = \h \cap \s^+ \subset \so(T_p M)$  and $\mathfrak{m}|_p$ spans $T_p M$. Hence, $\mathfrak{m}$ is a Killing algebra of a metric in the conformal class, acting locally transitively on $M$. 

\noindent \textbf{Case} $\boldsymbol{\alpha = \frac{1}{2}}$.  
It follows from Corollary \ref{Cor. ideal defining F} that  the foliation $\F^+$ is everywhere tangent to the subalgebra generated by $\n^+$, i.e.~to the nilpotent algebra $\n:= \n^+ \oplus\c$, where $\c:=[\g^\alpha, \g^\alpha] \subset \g^1$ is abelian, and which acts locally transitively on each leaf of $\F^+$.  
We will show that $\n$ is abelian, by  proving that $[\g^\alpha, \g^\alpha] = 0$, or in other words, we will show that $\n^+$ already is an abelian subalgebra.

For any subalgebra $\k \subset \g$, denote by $\overline{\k}$ its projection in $\overline{\g} := \g / \g^{\alpha+1}$.
By Lemma \ref{Lem. existence of foliation}, the $\F^+$-leaves are foliated by the radical foliation tangent to $\g^{\alpha+1}$.
For each point $q \in M$, define the quotient space $T_q$ (depending on $q$) of the $\F$-leaf through $q$, restricted to a neighborhood of $q$, by the radical foliation. Then $T_q$ is a conformally Riemannian manifold, homogeneous under the action of $\overline{\n}$.   
Denote by $\i_q \subset \overline{\n}$ the isotropy subalgebra at $q$. We will show that $\i_q$ consists precisely of the elements of $\overline{\n}$ that act trivially on $T_q$. For $q = p$, we have $\i_p = \overline{\c} \subset \overline{\g^1}$. Since $[\n, \g^1] \subset \g^{\alpha+1}$, it follows that $[\overline{\n}, \overline{\g^1}] = 0$, which shows that $\i_p=\overline{c} \subset \overline{\g^1}$ vanishes on all of $T_p$, hence acts trivially on $T_p$.
For $q \neq p$, the isotropy $\i_q$ is conjugate to $\i_p$ by an element of $\overline{\g}$, and hence also acts trivially on $T_q$. This shows that $\i_q$ is, in fact, an ideal of $\overline{\n}$ for any $q$. Moreover, the quotient $\overline{\n} / \i_q$ acts faithfully on $T_q$.
Next, observe that the metric on $T_p$ is conformally flat due to the action of $e^{tB}$: indeed, $B(p) = 0$, and $\ad_B$ acts on the tangent space of $T_p$ at $p$ as $\alpha I_{n-2}$. Consequently, the same property holds for $T_q$ for any $q$. In dimension $n - 2 \ge 3$, this implies that $\overline{\n} / \i_q$ embeds into $\mathfrak{o}(1, n - 1)$ as a nilpotent subalgebra. However, every nilpotent subalgebra of $\mathfrak{o}(1, n - 1)$ is abelian, so $\overline{\n} / \i_q$ must be abelian. In the case $n - 2 = 2$, the algebra $\overline{\n} / \i_q$ is $2$-dimensional and nilpotent, hence also abelian.
Assume now that there exist $x, y \in \g^\alpha$ such that $z := [x, y] \in \g^1$ is nonzero. Then there exists a point $q \in M$ such that $\overline{z}(q) \neq 0$, for otherwise $z$ would belong to $\g^{\alpha+1}$. Consequently, $z$ projects nontrivially onto $\overline{\n} / \i_q$, which yields a contradiction.  
Finally, $\n =\n^+= \g^\alpha \oplus \g^{\alpha+1}$ is an abelian subalgebra of $\g$, acting locally simply transitively on each leaf of $\F$. 

Set $\m:=   \n \oplus \g^{\alpha-1}= \g^{\alpha-1}\+  \g^\alpha \oplus \g^{\alpha+1}$. 
It follows from the discussion above and from (\ref{Eq. bracket})-(\ref{Eq. bracket_1}) that 
$[\m,\m]=[\g^{\alpha-1},\g^{\alpha}]\subset \h^{00}$. Thus, 
the ideal $\k$ generated by $\m$  satisfies $\k \cap \h \subset \so(T_pM)$.  Moreover, $\k|_p$ spans $T_p M$. Hence, $\k$ acts transitively on $M$, and there is a metric in the conformal class for which $\k$ is an algebra of Killing fields. 
\end{proof}

\subsection{Case $\h \subset \g^0 \oplus \g^{-1}$} 
In this case, we will proceed differently from the previous one.  
We will define, in a neighborhood of  $p \in M$,  a foliation that is not necessarily invariant under all local conformal transformations.
Although one could see whether this foliation extends globally, this is not required for our purposes. The strategy is then to find:

\begin{enumerate}
\item 
a subalgebra $\k\subset\g$ such that $\k$ acts locally transitively and with stabiliser $\k_p$ contained in $\so(T_pM)$. This ensures, as before, the existence of a metric in the conformal class with respect to which $\k$ is an algebra of Killing fields;
\item 
an ideal $\n$ of $\k$ that is of Heisenberg type whose center contains a one-dimensional subspace $\z$ such that  $\z.p$ is null and $\n.p=(\z.p)^\perp$. This ensures that $q \mapsto \n|_q$ defines (in a neighborhood of $p$) a distribution which is integrable, tangent to a $\k$-invariant foliation $\G$. This foliation is therefore everywhere lightlike, and $\n$ acts locally transitively on its leaves. 
\end{enumerate}

\begin{lemma}\label{Lem. (2) Heis subalgebra}
Assume $\alpha \in \{\tfrac{1}{2}, 1, 2\}$ and $\h \subset \g^0 \oplus \g^{-1}$,  and let 
 $\n:=\g^\alpha \oplus \g^{\alpha+1}$. Then the distribution 
$
q \mapsto \n|_q
$
defines a codimension-one lightlike foliation $\G$. 
Moreover, 
\begin{enumerate}
    \item $\n$ is of Heisenberg type,
    \item the null line bundle $\G^\perp$ is tangent to the central one-dimensional subspace $\g^{\alpha+1}$, and $(\g^{\alpha+1}|_p)^\perp = \n|_p$.
    \item there exists a metric in the conformal class whose algebra of Killing vector fields acts locally transitively and contains $\n$.
    \end{enumerate}
\end{lemma}

\begin{proof}
\noindent \textbf{Case} $\boldsymbol{\alpha = 2}$. Define 
$
\l := \g^{\alpha-1} \oplus \g^\alpha \oplus \g^{\alpha+1}
$.  This is a locally simply transitive subalgebra,  and $\n$ is an abelian ideal in $\l$.

\noindent \textbf{Case} $\boldsymbol{\alpha = 1}$. Here,  
$
\l := \g^{\alpha-1} \oplus \g^\alpha \oplus \g^{\alpha+1}
$  is a locally simply transitive subalgebra, and $\n$ is a $2$-step nilpotent ideal in $\l$, with center $\g^{\alpha+1}$. 

\noindent \textbf{Case} $\boldsymbol{\alpha = \frac{1}{2}}$. In this case, we have $[\g^\alpha, \g^{\alpha -1}] \subset \g^0$. One shows using the Jacobi identity that actually $[\g^\alpha, \g^{\alpha -1}] \subset  \h^{00}$. Thus, the algebra $\k$ generated by   $\g^{\alpha-1} \oplus \g^\alpha \oplus \g^{\alpha+1}$ is locally transitive and satisfies $\k_p \subset \so(T_p M)$. In addition, $\n$ is an abelian ideal in $\k$. 

Thus, in all cases, the lemma follows directly. 
\end{proof} 
 
\subsection{Proof of Proposition \ref{Prop. existence of foliation}} Item (1) follows from Lemma \ref{Lem. CF case},  Item (2) from Corollary \ref{Cor. ideal defining F} and Lemma \ref{Lem. (1) Heis subalgebra}, and Item (3) from Lemma \ref{Lem. (2) Heis subalgebra}.

\section{Proof of Theorem~\ref{maintheo}}

Assume that  $(M,g)$ is  locally conformally homogeneous and weakly essential, i.e. the pseudo-group 
of local conformal transformations acts transitively on $M$ and does not  preserve a Lorentzian metric in the conformal class of $g$. Then, the image $\H_p$ of the isotropy representation   at $p\in M$ contains a one-parameter subgroup $B^t$ as in Corollary \ref{essential-cor}.
The same holds for the 
 universal cover $(\widetilde{M},\widetilde{g})$, which has the same pseudo-group of local conformal transformations, and on which the Lie algebra $\g$ of global conformal  vector fields  acts transitively. 
 By Corollary \ref{para-cor},  either $(M,g)$ is conformally flat,  or the Lie algebra  $\h$ of $\H_{\widetilde{p}}$ at  $\tilde{p} \in \widetilde{M}$   
 stabilizes a null line in $T_{\tilde{p}}\widetilde{M}$. 
 By Proposition \ref{Prop. isotropy}, we may assume that the isotropy representation is injective. 
   Hence, we can  apply  
 Proposition \ref{Prop. existence of foliation}
to $\g$ and $\h$,   and we obtain that either the metric is conformally flat, or there exists locally a codimension-one lightlike foliation of  \textit{Heisenberg type}. In the latter case, Proposition \ref{pwkilling-prop} implies that the metric is locally conformal to a plane wave. Moreover, the existence of a locally transitive algebra of Killing vector fields, as given by Proposition \ref{Prop. existence of foliation}, ensures that this plane wave is isometrically locally homogeneous.\medskip

To  establish the last statement of Theorem \ref{maintheo}, we require the following result.

\begin{proposition}
Assume that $(M,[g])$ is not conformally flat. If $(M,[g])$ is locally conformally a plane wave, then there exists a homogeneous plane wave $X$ such that $(M,[g])$ is modeled on $(\Conf(X),X)$.     
\end{proposition}

In general, given a locally conformally homogeneous pseudo-Riemannian manifold $(M,[g])$, it is not necessarily true that there exists a homogeneous model to which $(M,[g])$ is locally conformal. Consider a point $p \in M$, and let $\mathfrak{g}$ be the Lie algebra of local conformal vector fields in a neighborhood of $p$, and $\mathfrak{h}$ the subalgebra of those vanishing at $p$.  
Let $G$ be the unique connected and simply connected Lie group with Lie algebra $\mathfrak{g}$, and let $H$ be the connected subgroup generated by $\mathfrak{h}$.  
Then, by \cite{Mostow}, the manifold $(M,[g])$ admits a homogeneous model if and only if $H$ is closed in $G$.  
In that case, $M$ is locally conformal to the homogeneous space $G/H$.

In our situation, when $(M,[g])$ is a non-conformally flat Lorentzian manifold locally conformal to a plane wave,  it follows from  \cite{Holland-Sparling} (see Section~\ref{pwbackground-sec}) that $\g$ is an $\R$-extension of the algebra of Killing fields $$\g_0 \cong (\R \oplus \k) \ltimes \mathfrak{heis},$$ where $\k$ is a subalgebra contained in the isotropy $\h$. More precisely, $$\g=(\R \oplus \R \oplus \k) \ltimes \mathfrak{heis}.$$

Let $G$ be the connected and simply connected Lie group with Lie algebra $\g$. We have 
$\g = \k \ltimes \mathfrak{r}, \quad$ with $\quad  \mathfrak{r} = (\R \oplus \R) \ltimes \mathfrak{heis}$
a solvable ideal. Then there is a  decomposition
$G \cong K \ltimes R$, 
where $K \leq G$ is the connected subgroup with Lie algebra $\k$ and $R \leq G$ is the connected subgroup with Lie algebra $ \mathfrak{r}$. Note that $R$ is solvable and simply connected.
Now, since $\k \subset \h$, we have 
$\h = \k \ltimes \h_1, \quad \h_1 := \h \cap  \mathfrak{r}$. 
Let $H \leq G$ denote the connected subgroup with Lie algebra $\h$. Then
$H = K \ltimes H_1 \subset K \ltimes R$, 
where $H_1 \leq R$ is the connected subgroup with Lie algebra $\h_1$.
By a classical fact, every connected Lie subgroup of a simply connected solvable Lie group is closed \cite[Corollary, p. 186]{Malcev}, \cite[Lemma 2]{H-C}. Hence $H_1$ is closed in $R$, and consequently $H$ is a closed subgroup of $G$.  
Therefore, the quotient space
$X = G / H$
is well-defined, and the manifold $M$ is locally modeled on $(G, X)$.

\section{On Penrose limits}
\label{Penrose-sec}

\subsection{Penrose limit: a conformal invariant} The aim of this  paragraph is to show that the Penrose limit associated to a Lorentzian metric is a conformal invariant. \medskip

Let $g$ be a Lorentzian metric, and $\gamma$ a null geodesic. To define the Penrose limit of $g$ along $\gamma$ (see \cite{Blau-PL}), we first write the metric in coordinates adapted to $\gamma$. There exist local coordinates $(u,v,x)$ such that:  

\begin{itemize}
    \item the curve $\gamma$ is given by $u \mapsto (u,0,0)$, and $u$ is an affine (geodesic) parameter of $\gamma$,
    \item the metric in these coordinates is written as
\begin{equation}\label{Eq-adapted coord} g = 2\d u\d v + a(u,v,x) \d v^2 + \sum_i 2 b_i(u,v,x) \d v dx_i + \sum_{i,j} c_{ij}(u,v,x) \d x_i \d x_j \end{equation}
\end{itemize}
In coordinates of the form (\ref{Eq-adapted coord}), the curves $u \mapsto (u,v_0,x_0)$ are null geodesics for any constants $v_0$ and $x_0$.  

Then, the Penrose limit of $g$ along $\gamma$ is defined as follows. For $\epsilon > 0$, set
$$
\Phi_\epsilon(u,v,x) = (u, \epsilon^2 v, \epsilon x), \quad g_\epsilon := \frac{1}{\epsilon^2} \Phi_\epsilon^* g.
$$
As $\epsilon \to 0$, the metric $g_\epsilon$ converges to a plane wave metric (in Rosen coordinates)
\begin{equation}\label{Eq-PL_g} \mathrm{PL}_g:= 2 \d u \d v  + \sum_{i,j} \overline{c}_{ij}(u)\d x_i\d x_j, \end{equation}
where
$$
\overline{c}_{ij}(u) := c_{ij}(u,0,0).
$$
We will now show that the Penrose limit is conformally invariant in the following sense.

\begin{proposition}\label{Penrose-prop}
Let $g$ be a Lorentzian metric, and let $\gamma$ be a null geodesic. Let $g_\sigma := e^\sigma g$ be a conformal change of $g$ (up to reparametrization, $\gamma$ is also a null geodesic of $g_\sigma$). Let $\mathrm{PL}_g$ (resp.\ $\mathrm{PL}_{g_\sigma}$) denote the Penrose limit of $g$ (resp.\ of $g_\sigma$) along $\gamma$. Then there exists a conformal diffeomorphism
$ \mathrm{PL}_g \to \mathrm{PL}_{g_\sigma}.$
\end{proposition}

\begin{proof}
Let $g_\sigma = e^{\sigma} g$, and set $K := e^{\sigma}$. To compute the Penrose limit of $g_\sigma$ along $\gamma$, we first write the metric in coordinates adapted to $\gamma$. For this, consider the map
$$ F : (u,v,x) \mapsto (f(u,v,x), v, x),$$ 
where $f$ is a smooth function satisfying
\begin{equation}\label{Eq}
\frac{\partial f}{\partial u} = K.
\end{equation}
Then $F$ is a diffeomorphism. Its inverse
$$G : (u,v,x) \mapsto (h(u,v,x), v, x)$$
satisfies
$$f(h(u,v,x), v, x) = u,
\qquad
(K \circ G)\, \frac{\partial h}{\partial u} = 1.$$
Then, the pullback metric $G^{*} g_\sigma$ takes the form
$$G^* g_\sigma
= 2\, du\, dv
+ A(u,v,x)\, dv^{2}
+ \sum_i 2 B_i(u,v,x)\, dv\, dx_i
+ \sum_{i,j} C_{ij}(u,v,x)\, dx_i\, dx_j,$$
where $A$, $B_i$, and $C_{ij}$ are smooth functions, and
$$C_{ij} := (K \circ G)\, (c_{ij} \circ G).$$
The Penrose limit of $g_\sigma$ along $\gamma$ is then given by
$$\mathrm{PL}_{g_\sigma}=
2\, du\, dv
+ \overline{C}_{ij}(u)\, dx_i\, dx_j,$$
where
$$\overline{C}_{ij}(u)=C_{ij}(u,0,0)=(K \circ G(u,0,0))\, (c_{ij} \circ G(u,0,0)).$$
Now, consider the diffeomorphism 
$$\phi: (u,v,x) \mapsto (f(u,0,0),v,x).$$
Using (\ref{Eq}), we compute 
$$
  \phi^* (\mathrm{PL}_{g_\sigma}) = K(u,0,0) (2\d u\d v + c_{ij}(u,0,0) \d x_i \d x_j)
 = K(u,0,0) \mathrm{PL}_g.  
$$
Hence, $\phi: \mathrm{PL}_{g} \to \mathrm{PL}_{g_\sigma}$ is a conformal diffeomorphism. 
\end{proof}

\begin{remark}
It appears from the proof above that the Penrose limit of $e^\sigma g$  along a null geodesic $\gamma$ of $g$ is given by $e^{\overline{\sigma}}\; \mathrm{PL}_g$, where $\overline{\sigma}(u):= \sigma(u,0,0)$ is the restriction of $\sigma$ to $\gamma$, and $\mathrm{PL}_g$ is the Penrose limit of $g$ along $\gamma$.    
\end{remark}

\subsection{Penrose limits of a plane wave}
Given a plane wave $(M,g)$ of dimension $n+2$, with parallel null vector field $\xi$, the Penrose limit taken along a null geodesic tangent to $\xi$ is the flat Minkowski space. On the other hand, if the plane wave is written in Rosen coordinates (\ref{Eq-PL_g}), the Penrose limit taken along the null geodesic $(u,0,0)$, which is transversal to $\xi$, coincides with the original plane wave metric.

Let $p \in M$, and let $C_p \subset T_p M$ denote the null cone at $p$. Then the isotropy group $H_p$ fixes  the null vector $\xi_p$, and acts transitively on $C_p \smallsetminus \R \xi_p$.
Consequently, up to isometry, there are exactly two (germs of) Penrose limits associated to $(M,p)$:
\begin{itemize}
    \item The Penrose limit taken along a null geodesic tangent to $\xi_p$. In this case, the resulting limit is flat Minkowski space.
    \item The Penrose limit taken along a null geodesic transversal to $\xi_p$. In this case, the resulting limit is
    the plane wave metric in some neighborhood of $p$ (the neighborhood may depend on the chosen geodesic). Thus, it is
    a plane wave embedded in the original plane wave. 
\end{itemize}
If the plane wave is locally homogeneous, the Penrose limit does not depend on the base point $p$. Hence, up to isometry, there are exactly two (germs of) Penrose limits associated to $M$. 

\begin{corollary}\label{Penrose-cor}
Let $(M, g)$ be a non-conformally flat Lorentzian manifold of dimension at least $3$ which is locally conformally homogeneous. Assume that $(M,g)$ is weakly essential. Then, $(M,g)$ 
 is locally conformally equivalent to its Penrose limit along some null geodesic.  
\end{corollary}
Note that the proof of Theorem~\ref{maintheo} reveals more details than those stated in this corollary. The assumption implies that there is a null line bundle that is invariant under the local conformal vector fields, which is spanned by the parallel vector field of the plane wave, and that $(M,g)$ is locally conformally equivalent to the Penrose limit along all null geodesics that are  \textit{transversal} to this null line bundle.

\appendix

 \section{Account on   rigid transformations groups}  \label{Gromov's Theory}
 \addtocontents{toc}{\setcounter{tocdepth}{1}}

For references, see the foundational work of Gromov~\cite{Gro}, the lecture notes by D'Ambra and Gromov~\cite{DG}, and those by Ballmann~\cite{Bal}.

\subsection{Gromov's theory of rigid transformation groups}

In Gromov's theory of rigid transformation groups, one encounters the following fundamental construction.  
Let $\sigma$ be a rigid geometric structure on a manifold $M$. For any point $p \in M$, denote by $\Iso^{\mathrm{Loc}}_p$ the group of germs of local diffeomorphisms fixing $p$ and preserving $\sigma$, and by $\Iso^k_p$ the infinitesimal isometry group of order $k$ at $p$.

Since everything is local, we may simplify notation by assuming $M = \R^n$ and $p = 0$.  
Let $\mathcal{D}^k_n$ denote the group of $k$-jets of germs of diffeomorphisms at~$0$.  
Two diffeomorphisms are identified if their Taylor expansions at $0$ agree up to order~$k$ (in any coordinate system).  
Thus, elements of $\mathcal{D}^k_n$ can be viewed as polynomial maps $\R^n \to \R^n$ of degree $\leq k$ vanishing at $0$, with composition given by polynomial composition truncated at order $k$.  
For $k=1$, $\mathcal{D}^1_n = \GL(n, \R)$.

Let $\mathcal{U}^k \subset \mathcal{D}^k_n$ denote the subgroup consisting of elements whose first derivative at $0$ is the identity. This subgroup is unipotent, and one has the semidirect product decomposition
$\mathcal{D}^k_n = \GL(n,\R) \ltimes \mathcal{U}^k$. 
Let $E$ be the vector space of polynomial maps $\R^n \to \R^n$ of degree $\leq k$ that vanish at $0$. The group $\mathcal{D}^k_n$ acts naturally on $E$ on the right by composition:
$$\mathcal{D}^k_n \times E \to E, \quad (f,P) \mapsto P \circ f.$$
This action defines a faithful representation of $\mathcal{D}^k_n$ on $E$, given by
$$f \mapsto M_f^\top \in \End(E), \quad M_f(P) = P \circ f.$$
The map $f \mapsto M_f^\top$ is a group homomorphism, since
$M_{f_1 \circ f_2} = M_{f_2} \circ M_{f_1}$. 
Up to conjugation, this representation sends the unipotent subgroup $\mathcal{U}^k$ into the group of upper triangular unipotent matrices; in particular, its image lies in $\GL(E)$.
Writing an element $f \in \mathcal{D}^k_n$ as $f = (A_f, f^u)$ with respect to the decomposition
$\mathcal{D}^k_n = \GL(n,\R) \ltimes \mathcal{U}^k$,
we may factor $f$ as
$f = f_1 \circ f_2, \quad f_1 := (\1, f^u), \quad f_2 := (A_f,1)$. 
Since $A_f \in \GL(n,\R)$, the map $f_2$ is invertible, and hence so is $M_{f_2}$. It follows that the image of $\mathcal{D}^k_n$ under the above representation is contained in $\GL(E)$, hence 
$\mathcal{D}^k_n$ embeds as a subgroup of  $\GL(E)$. 
Moreover, this subgroup is determined by polynomial equations in the matrix entries of $\GL(E)$, and is therefore an algebraic subgroup. Finally, the algebraic structure of $\mathcal{D}^k_n$ is naturally defined via this embedding.

\subsection*{Infinitesimal isometries}
We can similarly define $k$-jets of the geometric structure $\sigma$.  
A local diffeomorphism $\phi$ near $0$ is called a $(k+1)$-infinitesimal isometry if the pullback $\phi^*\sigma$ and $\sigma$ have the same $k$-jet at~$0$.  
To make this explicit, suppose $\sigma$ is a pseudo-Riemannian metric.  
Then $\sigma$ can be written as
$x \in \R^n \mapsto (\sigma_{ij}(x)) \in \R^{n(n+1)/2}$.
The condition that $\phi$ is a $(k+1)$-infinitesimal isometry means that
$\phi^*\sigma - \sigma$
vanishes up to order~$k$, i.e. their Taylor expansions at $0$ agree to that order.
This condition depends only on the $(k+1)$-jet of $\phi$, and hence defines a subgroup
$$\Iso^{k+1}_0 \subset \mathcal{D}^{k+1}_n,$$
as the polynomial maps $\phi_{k+1}:\R^n \to \R^n$ of degree $\leq k+1$ vanishing at $0$, and such that  $\phi_{k+1}^* \sigma - \sigma$ vanishes up to order $k$. 
Also, observe that only the $k$-jet of $\sigma$ is involved here.   
It is then clear that the condition $\phi^*_{k+1}\sigma = \sigma$ up to order~$k$ is algebraic, since it is the zero set of a finite number of polynomials on $\mathcal{D}^{k+1}_n$. 
Thus, $\Iso^{k+1}_0$ is an algebraic subgroup of $\mathcal{D}^{k+1}_n$.

In the case of a conformal structure determined by a metric $g$, a $(k+1)$-infinitesimal isometry is defined by requiring that $g$ and $\phi^* g$ be conformal up to order~$k$, i.e. their $k$-Taylor developments are proportional (as vectorial polynomial maps).  
Equivalently, for all indices $i,j,l,m$,
$(\phi^*g)_{ij} g_{lm} - (\phi^*g)_{lm} g_{ij}$ 
vanishes up to order~$k$.

\subsection*{Gromov's Theorem}
Returning to the manifold~$M$, there are natural morphisms
\begin{equation}\label{Eq: Gromov}
  \Iso^{\mathrm{Loc}}_p \to \Iso^k_p.   
\end{equation}
Gromov's theory asserts that there exists an integer~$k$,  depending on the order of the geometric structure  and on $\dim M$, such that, for $p$ generic in  $M$, these morphisms are bijective  
(see \cite{Gro},  paragraph  0.3.A, p. 68, and Theorem 1.6.F, p. 86).  
In other words, for sufficiently large~$k$ and generic~$p$, every $k$-infinitesimal isometry is the $k$-jet of a genuine local isometry.

For such $p$ and $k$, the derivative representation 
$$ 
 \rho: \Iso^{\mathrm{Loc}}_p \to \GL(T_pM)
$$ 
coincides with the restriction to the algebraic subgroup $\Iso_p^k$ of $\mathcal{D}_n^k$ of the natural projection $\mathcal{D}_n^k = \GL(n,\R) \ltimes \mathcal{U}^k \to \GL(n,\R)$. It is therefore an algebraic homomorphism.  Let $p \in M$ and $k$ be such that (\ref{Eq: Gromov}) is an isomorphism. Then, by Lemma \ref{Lem: algebraic image}, the image of the derivative representation  contains the identity  component (for the Hausdorff topology) of its Zariski closure in  $\GL(n,\R)$. In the locally homogeneous case, this holds at every point.

Furthermore, all these local or infinitesimal groups carry natural Lie group structures, and bijective morphisms between them are Lie group isomorphisms. Consequently, both $\Iso^{\mathrm{Loc}}_p$ and its image under the derivative representation have finitely many connected components.

\subsection{Singer's Theorem}
\label{Singer}

Singer's Theorem concerns Riemannian metrics and states that infinitesimal homogeneity implies local homogeneity.

Let $(N, h)$ be a Riemannian manifold with curvature tensor~$R$.  
We say that $(N, h)$ is $k$-infinitesimally homogeneous if $R$ and its covariant derivatives $\nabla R, \ldots, \nabla^k R$ are the same at all points.  
That is, for any $p, q \in N$, there exists a linear isomorphism
$A_{pq} : T_p N \to T_q N$
sending $R_p, \nabla R_p, \ldots, \nabla^k R_p$ to $R_q, \nabla R_q, \ldots, \nabla^k R_q$.
Singer's Theorem asserts that there exists an integer $k = k(\dim N)$ such that if $(N, h)$ is $k$-infinitesimally homogeneous, then $(N, h)$ is locally homogeneous.  
Moreover, such an $A_{pq}$ is the derivative of a local isometry sending $p$ to $q$.  
In particular, the image of the local isotropy group under the derivative representation,
$$\Iso^{\mathrm{Loc}}_p \to \GL(T_pN),$$ 
is exactly the subgroup of $\GL(T_pN)$ preserving the tensors
$$ R_p, \nabla R_p, \ldots, \nabla^k R_p,$$
and is thus an algebraic subgroup.

One section of Ballmann's lecture notes \cite{Bal} provides an adaptation of Singer's Theorem to general rigid geometric structures following Gromov's theory.  
In the locally homogeneous case, it shows in a  clear way  that  the local isotropy is algebraic.

\subsection{Cartan connections approach} Pecastaing's article \cite{Pecastaing} yields a simplified approach of Gromov's theory for geometric structures given by a Cartan connection. This 
 covers in particular the conformal pseudo-Riemannian structures.  In particular, it contains an analogue of Singer's local homogeneity theorem for Cartan geometries (\cite[Theorem~1.1]{Pecastaing}).

 Briefly, a Cartan geometry modeled on $G/P$ consists of a $P$-principal bundle $\hat{M} \to M$ equipped with a $1$-form, called the Cartan connection 
$$
  \omega : T{\hat{M}} \to \mathfrak{g},
$$
with values  in the Lie algebra $\g$ of $G$. This Cartan connection induces a 
 parallelism  as follows: given a basis $(X_1,\dots,X_r)$ of $\g$, then $(\omega^{-1}(X_1), \dots, \omega^{-1}(X_r))$ is a parallelism of $T \hat{M}$, identifying $T \hat{M}$ with $\hat{M} \times \g$. 
 This parallelism is $P$-equivariant and coincides with left invariant parallelism of $P$ along the vertical.  
 
The curvature of $\omega$,
$$ 
  K = d\omega - [\omega, \omega],
$$
is the obstruction of parallel vector fields on $\hat{M}$ to form a Lie algebra isomorphic to~$\mathfrak{g}$.  
Via the parallelism, $K$ can be seen as a $P$-equivariant vectorial map
$$
  \kappa : \hat{M} \to  \mathsf{Hom}(\Lambda^2 \mathfrak{g}, \mathfrak{g}),
$$
called the curvature map. 
  It turns out that $\kappa$  takes values in a small $P$-invariant subspace $\mathcal{W}^0 \subset \mathsf{Hom}(\Lambda^2 \mathfrak{g}, \mathfrak{g})$. We can now differentiate $\kappa$, getting a map $D \kappa : T
\hat{M} \to \mathcal{W}^0$. 
Similarly, the parallelism allows one to express  $D \kappa$ as a vectorial map 
 $D \kappa: \hat{M} \to \mathcal{W}^1$, with $\mathcal{W}^1 := \mathsf{Hom}(\g,\mathcal{W}^0)$. 
  One can do the same
for iterated derivatives and define vectorial maps for $i \geq 0$
$$ 
  \kappa^i : \hat{M} \to \mathcal{W}^i, \qquad \kappa^0 := \kappa, \qquad \kappa^i := D \kappa^{i-1}, \quad \mathcal{W}^i := \mathsf{Hom}(\g, \mathcal{W}^{i-1}).
$$ 
These $\mathcal W^i$ are  endowed with a natural linear $P$-action, and $\kappa^i$
is $P$-equivariant with respect to this action.

A point in $\hat{M}$ lying over a point $x \in M$ is denoted by $\hat{x}$.  By $P$-equivariance, $\kappa^i$ sends a $P$-fiber  of $\hat{M} \to M$ above $x \in M$  to a $P$-orbit $C^i_x = \kappa^i (P \cdot \hat{x}) \subset \mathcal W^i$.  Then, $k$-infinitesimal  homogeneity means that 
 $C_x^i$ does not depend on $x \in M$ for all $i \leq k$, i.e. $ \kappa^i(P \cdot \hat{x}) = C_x^i = C^i$, for all $x \in M$.  

To put all these curvatures together, for any $l$, consider $$\kappa_l = (\kappa^0, \kappa^1,\ldots, \kappa^l) : \hat{M} \to \mathcal{Y}^l := \mathcal{W}^0 \times \mathcal{W}^1 \times \ldots \times  \mathcal{W}^l.$$

\subsection*{The pseudogroup action and local Killing fields} 
Henceforth, $\hat{M}$ is the Cartan bundle associated to a conformal pseudo-Riemannian manifold $M$. 
Let $\mathbf{P}$ denote the pseudogroup of local conformal transformations of $M$.  
It acts naturally on $\hat{M}$ (in fact freely and properly) by preserving the $P$-fibration. Moreover, this action preserves all the maps $\kappa_l$; that is, $\mathbf{P}$-orbits in $\hat{M}$ are contained in $\kappa_l$-levels. 
Local $\omega$-Killing  fields of the Cartan bundle $\hat{M}$ are also defined. A local $\omega$-Killing field  is a local vector field of $\hat{M}$ preserving the Cartan connection  $\omega$. An $\omega$-$\mathrm{Kill}^{loc}$-orbit of $\hat{x} \in \hat{M}$ is the set of points of $\hat{M}$ that can be reached from $\hat{x}$ by a finite sequence of flows of local $\omega$-Killing fields. 

There is a one-to-one correspondence between local conformal vector fields of $M$ and $\omega$-Killing fields of $\hat{M}$.  The local flow of an $\omega$-Killing field $\hat{X}$ preserves $\kappa_l$, which implies that $\hat{X}$ belongs to $\ker (D_{\hat{x}} \kappa_l)$ at every point $\hat{x}$ where $\hat{X}$ is defined, and for every $l \geq 0$. If $M$ is  locally conformally homogeneous, i.e. if ${\bf P}$ acts transitively on $M$, then it is $l$-infinitesimally homogeneous for any order $l$, and the map $\kappa_l$ has constant rank. Its level sets are therefore smooth submanifolds of $\hat{M}$   that project surjectively on $M$. 
In particular,  the $P$-invariant vector subspaces $\ker (D_{\hat{x}} \kappa_l) \subset T_{\hat{x}} \hat{M}$ project surjectively onto $T_x M$.
 
The point  proved by V. Pecastaing \cite{Pecastaing} is that, on a  $P$-invariant dense open subset $\hat{\Omega} \subset \hat{M}$, and  for $m$ big enough (depending on the dimension), $\kappa_l$-levels, for $l \geq m$,  coincide with the 
 $\bf P$-orbits in $\hat{\Omega}$.   

It is also shown in \cite{Pecastaing} that the evaluation of $\omega$-Killing fields generate all the tangent  space of the $\bf P$-orbits in $\hat{\Omega}$. 
Indeed, an integrability theorem shown in \cite{Pecastaing} (see also \cite[Theorem 2.2 and Annex A]{F-AnnexA}) yields the following:  for $m$ big enough, any vector $\hat{v} \in \ker(D_{\hat{x}} \kappa_l)$, $\hat{x} \in \hat{\Omega}$,  
is the evaluation of an actual local $\omega$-Killing field around $\hat{x}$. It follows that the $\omega$-$\mathrm{Kill}^{loc}$-orbits of $\hat{M}$  are exactly the connected components of the ${\bf P}$-orbits (since they have the same tangent spaces). 
When $M$ is locally conformally homogeneous, this set $\hat{\Omega}$ is the whole $\hat{M}$.

Thus, if $M$ is locally conformally homogeneous, for any vector $v \in T_x M$, and a lift $\hat{v} \in \ker(D_{\hat{x}} \kappa_l) \subset T_{\hat{x}} \hat{M}$, there exists an $\omega$-Killing field $\hat{V}$ such that $V(\hat{x})=\hat{v}$. This $\hat{V}$  projects to a Killing field $V$ on $M$ satisfying $V(x)=v$.
It follows that, when ${\bf P}$ acts transitively on $M$, the pseudo-algebra of local conformal vector fields also acts transitively on $M$; that is, for any $p \in M$, the evaluation map for local conformal vector fields is surjective on $T_p M$.

\subsection*{Local isotropy via the pseudogroup action} Since the induced action of $\bf P$ on $\hat{M}$ preserves the $P$-fibration,   if $f \in \bf P$  sends a point in a $P$-fiber  to another point in the same $P$-fiber, then $f$ preserves
 this 
$P$-fiber. 
Similarly,  if a local Killing 
 vector field is somewhere tangent to a $P$-fiber,   then it is everywhere  tangent to it.  
 
Since $\bf P$ acts freely on $\hat{M}$, one can identify the local isotropy group $\Iso^{Loc}_x$  of a point $x \in M$ with $O  \cdot \hat{x}$, 
 for any  $\hat{x} \in \hat{M}$ above $x$,   where 
 \[
  O \cdot \hat{x} = (\mathbf{P} \cdot \hat{x}) \cap (P \cdot \hat{x}),
\]
 i.e. the intersection of  the $\bf P$-orbit of $\hat{x}$
 with its $P$-fiber.  But this is also the intersection of $P \cdot \hat{x}$ with the $\kappa_m$-level of $\hat{x}$.  This can be described as follows. Consider 
 $$\kappa_m^x: P \cdot \hat{x} \to \mathcal{Y}^m,$$ the restriction of $\kappa_m$ to the $P$-fiber above $x$. This is a $P$-equivariant map whose  image is the $P$-orbit $C^m = P \cdot \kappa_m( \hat{x})$. 
 Let $R \subset P$ denote the stabilizer of $\kappa_m (\hat{x})$. Then, the $\kappa^x_m$-level of $\hat{x}$ in $P \cdot \hat{x}$
 is 
 $$(\kappa^x_m)^{-1} (\kappa_m(\hat{x})) = R \cdot \hat{x}.$$ 
 
Hence, the local isotropy group $\Iso^{\mathrm{Loc}}_x$ is identified with the subgroup $R \subset P$. As a stabilizer of the point $\kappa_m(\hat{x})$ of $\mathcal{Y}^m$, $R$ is an algebraic subgroup 
 of $P$ (for conformal pseudo-Riemannian structures, $P$ is an algebraic group).  From the previous statement that local Killing fields generate the tangent spaces of 
 $\bf P$-orbits, it follows that the Lie algebra of $\Iso_x^{Loc}$ consists precisely of local Killing fields vanishing at $x$.

 \begin{remark}  As mentioned above,  for locally homogeneous  Riemannian  metrics, Singer's Theorem gives an explicit description of the local isotropy group at $p$ as the subgroup of $\GL(T_pM)$ preserving the curvature tensor and its covariant derivatives.    In the conformal case, the local isotropy rather naturally lies in the group of 
 2-jets of diffeomorphisms, which is something delicate for manipulation. One can however consider the image of $\Iso_p^{Loc}$  by the derivative representation and 
  get a  subgroup $\bf H_p$  of $\GL(T_pM)$.  It is interesting to characterize $\bf H_p$ analogously to Singer's Theorem case. Naturally, the Weyl tensor must be involved, but higher order conformally invariant tensors are not obvious to define.  
 \end{remark}

\bibliographystyle{abbrv}

\begin{thebibliography}{1}


\bibitem{Alekseevskii72}
D.~V.~Alekseevski\u{i}.
\newblock Groups of conformal transformations of {R}iemannian spaces.
\newblock {\em Mat. Sb. (N.S.)}, 89(131):280--296, 356, 1972.


\bibitem{AlekseevskyGalaev25}
D.~V.~Alekseevsky and A.~S. Galaev.
\newblock Conformally homogeneous lorentzian spaces.
\newblock {\em Communications in Contemporary Mathematics}, 0(0):2550082, 2025.



\bibitem{Bal} W. Ballmann, Geometric structures, Lecture notes available 
\url{https://people.mpim-bonn.mpg.de/hwbllmnn/archiv/geostr00.pdf}


\bibitem{Blau-PL} Blau, M. (2011). Plane waves and Penrose limits. Lecture notes for the ICTP school on mathematics in string and field theory (June 2-13 2003).


\bibitem{blau-oloughlin03}
M.~Blau and M.~O'Loughlin.
\newblock Homogeneous plane waves.
\newblock {\em Nuclear Phys. B}, 654(1-2):135--176, 2003.




\bibitem{DG} G. d'Ambra and M. Gromov, Lectures on the transformation groups: geometry and dynamics, J. Differential Geom. Suppl. 1 (1991), 19 -111.



\bibitem{olmos-discala01}
A.~J. Di~Scala and C.~Olmos, \emph{The geometry of homogeneous
  submanifolds of hyperbolic space}, Math. Z. \textbf{237} (2001), no.~1,
  199--209.


\bibitem{EJ09} Eberlein, P. and Jablonski, M., 2009. Closed orbits of semisimple group actions and the real Hilbert-Mumford function. Contemporary Mathematics, 491, p.283.


\bibitem{Ferrand96}
J.~Ferrand.
\newblock The action of conformal transformations on a {R}iemannian manifold.
\newblock {\em Math. Ann.}, 304(2):277--291, 1996.

\bibitem{Frances08}
C.~Frances.
\newblock Essential conformal structures in {R}iemannian and {L}orentzian
  geometry.
\newblock In {\em Recent developments in pseudo-{R}iemannian geometry}, ESI
  Lect. Math. Phys., pages 231--260. Eur. Math. Soc., Z\"urich, 2008.

\bibitem{frances12}
C.~Frances.
\newblock About pseudo-{R}iemannian {L}ichnerowicz conjecture.
\newblock {\em Transform. Groups}, 20(4):1015--1022, 2015.

\bibitem{Frances2012}
C.~Frances.
\newblock
D\'eg\'enerescence locale des transformations conformes pseudo-riemanniennes, Ann. Inst. Fourier 62 (2012), no. 5, 1627-1669.

\bibitem{F-AnnexA} C. Frances, Variations on Gromov’s open-dense orbit theorem. Bulletin de la Société mathématique de France 146.4 (2018): 713-744.


\bibitem{FM2013} C. Frances and K. Melnick, Formes normales pour les champs conformes pseudo-riemanniens, Bull. Soc. Math. France 141 (2013), no. 3, 377-421.

\bibitem{GlobkeLeistner16}
W.~Globke and T.~Leistner.
\newblock Locally homogeneous pp-waves.
\newblock {\em J. Geom. Phys.}, 108:83--101, 2016.

\bibitem{Gro}
M. Gromov. Rigid transformations groups. In G\'eom\'etrie diff\'erentielle (Paris, 1986), volume 33 of
Travaux en Cours, pages 65 -139. Hermann, Paris, 1988.


\bibitem{HPW} M. Hanounah, L. Mehidi, and A. Zeghib. On homogeneous plane waves. Journal of Mathematical Physics, volume 66 (2025).


\bibitem{H-C} Harish-Chandra. “On Faithful Representations of Lie Groups.” Proceedings of the American Mathematical Society 1, no. 2 (1950): 205–10. https://doi.org/10.2307/2031923.



\bibitem{Holland-Sparling} J. Holland and G. Sparling. Sachs equations and plane waves II: Isometries and conformal isometries. arXiv preprint arXiv:2405.12748, 2024.


\bibitem{Hum} Humphreys, James E. Linear algebraic groups. Vol. 21. Springer Science and Business Media, 2012.

\bibitem{LedgerObata70}
A.~J. Ledger and M.~Obata.
\newblock Compact riemannian manifolds with essential groups of
  conformorphisms.
\newblock {\em Trans. Amer. Math. Soc.}, 150:645--651, 1970.

\bibitem{LeistnerTeisseire22}
T.~Leistner and S.~Teisseire.
\newblock Conformal transformations of {C}ahen-{W}allach spaces.
\newblock {\em Ann. Inst. Fourier (Grenoble)}, 75(5):2147--2187, 2025.

\bibitem{Lelong-Ferrand71}
J.~Lelong-Ferrand.
\newblock Transformations conformes et quasi-conformes des vari\'et\'es
  riemanniennes compactes (d\'emonstration de la conjecture de {A}.
  {L}ichnerowicz).
\newblock {\em Acad. Roy. Belg. Cl. Sci. M\'em. Collect. 8(2)}, 39(5):44, 1971.


\bibitem{Malcev} A. Malcev, “On the theory of the Lie groups in the large”, Sb. Math., 58:2 (1945)


\bibitem{Mostow} George D. Mostow : The extensibility of local Lie groups of transformations and groups on surfaces. Ann. of Math., 52:606–636, 1950.

\bibitem{Obata71}
M.~Obata.
\newblock The conjectures on conformal transformations of {R}iemannian
  manifolds.
\newblock {\em J. Differential Geometry}, 6:247--258, 1971.
%


\bibitem{Pal56} Palais, Richard S. A global formulation of the Lie theory of transformation groups. Mem. Amer. Math. Soc., (22), 1957.


\bibitem{Pecastaing} V.~ Pecastaing, On two theorems about local automorphisms of geometric structures
Tome 66, no 1 (2016), p.~175 - 208.








\bibitem{Perr} Perrin, N. Linear algebraic groups. Hausdorff Center for Mathematics, Universitat Bonn (2015).


\end{thebibliography}
\providecommand{\MR}[1]{}\def\cprime{$'$} \def\cprime{$'$} \def\cprime{$'$}

\end{document}